\documentclass[12pt]{amsart}
\usepackage[T1]{fontenc}
\usepackage[latin9]{inputenc}
\usepackage{geometry}
\usepackage{mathrsfs}
\usepackage{amstext}
\usepackage{amsthm}
\usepackage{amssymb}
\usepackage{bbm}
\usepackage{graphicx}
\usepackage{subcaption}
\usepackage[table]{xcolor}

\usepackage[svgnames]{xcolor}
\usepackage[bookmarksnumbered=true]{hyperref} 
\hypersetup{
	colorlinks = true,
	linkcolor = Blue,
	anchorcolor = blue,
	citecolor = Green,
	filecolor = blue,
	urlcolor = FireBrick
}

\usepackage{algorithm}
\usepackage{algorithmic,eqparbox,array}

\usepackage{aliascnt}
\usepackage[nameinlink]{cleveref}

\makeatletter
\numberwithin{equation}{section}
\numberwithin{figure}{section}

\theoremstyle{plain}
\newtheorem{theorem}{Theorem}[section]

\theoremstyle{definition}
\newaliascnt{definition}{theorem}

\aliascntresetthe{definition}
\crefname{definition}{Definition}{Definitions}

\theoremstyle{definition}
\newaliascnt{question}{theorem}

\aliascntresetthe{question}
\crefname{question}{Question}{Questions}

\theoremstyle{definition}
\newaliascnt{remark}{theorem}
\newtheorem{remark}[remark]{Remark}
\aliascntresetthe{remark}
\crefname{remark}{Remark}{Remarks}

\theoremstyle{plain}
\newaliascnt{corollary}{theorem}
\newtheorem{corollary}[corollary]{Corollary}
\aliascntresetthe{corollary}
\crefname{corollary}{Corollary}{Corollaries}

\theoremstyle{plain}
\newaliascnt{lemma}{theorem}
\newtheorem{lemma}[lemma]{Lemma}
\aliascntresetthe{lemma}
\crefname{lemma}{Lemma}{Lemmas}

\theoremstyle{plain}
\newaliascnt{proposition}{theorem}
\newtheorem{proposition}[proposition]{Proposition}
\aliascntresetthe{proposition}
\crefname{proposition}{Proposition}{Propositions}

\theoremstyle{definition}
\newaliascnt{example}{theorem}
\newtheorem{example}[example]{Example}
\aliascntresetthe{example}
\crefname{example}{Example}{Examples}

\theoremstyle{definition}
\newaliascnt{assumption}{theorem}
\newtheorem{assumption}[assumption]{Assumption}
\aliascntresetthe{assumption}
\crefname{assumption}{Assumption}{Assumptions}

\theoremstyle{definition}
\newaliascnt{problem}{theorem}

\aliascntresetthe{problem}
\crefname{problem}{Problem}{Problems}

\newtheorem*{definition*}{Definition}
\newtheorem*{remark*}{Remark}
\newtheorem*{example*}{Example}

\renewenvironment{proof}[1][\proofname]{\medskip \noindent {\bfseries #1. }}{\hfill \qedsymbol\medskip}

\MakeRobust{\ref}
\newcommand{\labeltext}[2]{
\@bsphack
\csname phantomsection\endcsname
\def\@currentlabel{#1}{\label{#2}}
\@esphack
}

\usepackage{scalerel}
\usepackage[usestackEOL]{stackengine}
\def\dashint{\,\ThisStyle{\ensurestackMath{%
  \stackinset{c}{.2\LMpt}{c}{.5\LMpt}{\SavedStyle-}{\SavedStyle\phantom{\int}}}%
  \setbox0=\hbox{$\SavedStyle\int\,$}\kern-\wd0}\int}

\usepackage{enumitem}

\DeclareRobustCommand{\SkipTocEntry}[5]{}

\newcommand{\mR}{\mathbb{R}}   
\newcommand{\mN}{\mathbb{N}}   
\newcommand{\mK}{\mathbb{K}}   

\newcommand{\abs}[1]{\lvert #1 \rvert}  
\newcommand{\norm}[1]{\lVert #1 \rVert}  
\newcommand{\br}[1]{\langle #1 \rangle}  

\newcommand{\mD}{\mathscr{D}}

\newcommand{\mF}{\mathscr{F}}

\newcommand{\mL}{\mathcal{L}}

\newcommand{\rmd}{\mathrm{d}}
\newcommand{\sfd}{\mathsf{d}}

\makeatother

\usepackage{babel}

\usepackage{orcidlink}

\begin{document}

\title[A variational problem of species population density]{Existence and regularity of minimizers for a variational problem of species population density} 

\begin{sloppypar}

\begin{abstract}
We study a variational problem motivated by models of species population density in a nonhomogeneous environment. We first analyze local minimizers and then establish the existence and uniqueness of a global minimizer, along with its properties. We further investigate the structure of the saturated region, where the population attains its maximal density, using free boundary methods. For local minimizers, we apply recent results of Salo and Shahgholian, while for global minimizers, we establish a connection with the classical no-sign obstacle problem. Explicit radially symmetric solutions and heuristic numerical simulations illustrate our theoretical findings and provide insight into spatial saturation patterns in population models. 
\end{abstract}


\author[Kow]{Pu-Zhao Kow\,\orcidlink{0000-0002-2990-3591}}
\address{Department of Mathematical Sciences, National Chengchi University, Taipei 116, Taiwan \& Kanazawa University, Kakuma, 920-1192 Kanazawa, Japan}
\email{pzkow@g.nccu.edu.tw} 

\author[Kimura]{Masato Kimura\,\orcidlink{0000-0001-6612-1954}}
\address{Faculty of Mathematics and Physics, Kanazawa University, Kakuma, Kanazawa 920-1192, Japan}
\email{mkimura@se.kanazawa-u.ac.jp}

\author[Ohtsuka]{Hiroshi Ohtsuka\,\orcidlink{0000-0002-4473-4255}}
\address{Faculty of Mathematics and Physics, Kanazawa University, Kakuma, Kanazawa 920-1192, Japan}
\email{ohtsuka@se.kanazawa-u.ac.jp}

\maketitle

\tableofcontents{}





\subsection*{Funding} 
\addtocontents{toc}{\SkipTocEntry} 

Kow was supported by the National Science and Technology Council of Taiwan (Grant Numbers NSTC 112-2115-M-004-004-MY3, NSTC 115-2628-M-004-001) and by the National Center for Theoretical Sciences of Taiwan. Kimura was partially supported by JSPS KAKENHI (Grant Numbers JP24H00184, JP25K00920). Ohtsuka was partially supported by JSPS KAKENHI (Grant Number JP24K06794). 

\subsection*{Acknowledgments}
\addtocontents{toc}{\SkipTocEntry}

This project originated from Kow's three visits to the co-authors at Kanazawa University in March and August 2025, and January 2026, all generously supported by the university.

\subsection*{Notations} 
\addtocontents{toc}{\SkipTocEntry}

All formulas in the Digital Library of Mathematical Functions (\texttt{DLMF}, \url{https://dlmf.nist.gov/}), which is maintained by the National Institute of Standards and Technology of U.S. Department of Commerce (\url{https://www.nist.gov/}), can be also found in \cite{olver2010nist}. 
Throughout this paper, whenever a function $u \in W_{\mathrm{loc}}^{1,p}$ admits a continuous representative, we identify $u$ with that representative and regard it as an element of $C^0$.
For the precise definitions of $C^{m}(\Omega)$ and $C^{m,\alpha}(\Omega)$, we refer the reader to \cite{GT01Elliptic}. 
The following notations are provided for clarity, note that their usage may differ across mathematical disciplines: 
$J_{\upsilon}$, $Y_{\upsilon}$, $I_{\upsilon}$, and $K_{\upsilon}$ are Bessel functions (\Cref{sec:Bessel}).
$\Omega^{\complement} := \mR^{d}\setminus \Omega$ for any set $\Omega\subset\mR^{d}$.
$B_{R}(x_{0}):=\{x\in\mR^{d} : \abs{x-x_{0}}<R\}$ for all $R>0$ and $x_
{0}\in\mR^{d}$. We denote $B_{R}:=B_{R}(0)$. 
$V_\pm(x) := \max\{\pm V(x),0\}$.
We use ``$\sup$'' and ``$\inf$'' to denote the essential supremum and infimum, respectively. 
$\dashint_{\partial B_{r}} u \,\rmd S:=(d\omega_{d}r^{d-1})^{-1}\int_{\partial B_{r}} u \, \rmd S$ denotes the average of $u$ over $\partial B_{r}$, where $\omega_{d}:=\frac{\pi^{d/2}}{\Gamma(1+d/2)}$ is the volume of the unit ball in $\mR^{d}$. We denote $\mL^{d}$ be the Lebesgue measure on $\mR^{d}$. 
For any nonempty closed set $A\subset\mR^{d}$, we define 
\begin{equation*}
\sfd_{A}(x) := \min_{y\in A} \abs{x-y} \quad \text{for all $x\in\mR^{d}$.} 
\end{equation*}

\section{Introduction}

Let $d\ge 2$ be an integer. We consider the population density of a certain species in a nonhomogeneous environment modeled by $\mR^{d}$, assuming that the species migrates from afar (that is, from ``infinity''). Fisher \cite{Fisher1937} and the celebrated trio Kolmogorov, Petrovskii, and Piscunov \cite{KPP1937} investigated diffusion models in population dynamics as early as 1936. Skellam \cite{Skellam1951PopulationDynamics} conducted the first systematic study of the role of diffusion in population biology in 1951; see also \cite{Aronson1985SkellamRevisit,GM1977Population}. Skellam proposed a rational model for the spatial diffusion of a biological population based on the partial differential equation.
\begin{equation*}
\partial_{t}u = \Delta u + \sigma(u), 
\end{equation*}
where $u$ denotes the density of the population and $\sigma(u)$ represents the supply of the population. This theory is based on the assumption that dispersal results from the random movement of individuals. 

This paper investigates the scenario in which the dynamics of the population attains equilibrium. 
Let $u=u(x)\ge 0$ denote the density of the population at each point $x\in\mR^{d}$, subject to a finite carrying capacity: $u(x)\le 1$ for all $x\in\mR^{d}$.
The heterogeneity of the environment is described by a potential function $V=V(x)$, that is, $\sigma(u)=-Vu$, where $V(x)>0$ indicates an unfavorable (or ``bad'') environment and $V(x)<0$ indicates a favorable (or ``good'') environment for the species.

Suppose that the species exhibit diffusive movement, then the equilibrium states can be modeled by the following minimization problem: 
\begin{equation}
\begin{aligned}
& \text{minimizing} && \mF(u)\equiv\mF_{V}(u):=\int_{\mR^{d}} \left( \abs{\nabla u}^{2} + V(x)\abs{u}^{2} \right) \,\rmd x \text{\quad (globally/locally)} \\ 
& \text{subject to} && \mK := \{u \in H^{1}(\mR^{d}) : 0 \le u \le 1\}.
\end{aligned} \label{eq:main1} \tag{P} 
\end{equation}

\begin{definition*} 
A function $u_{*}\in\mK$ is called a \emph{global minimizer of \eqref{eq:main1} in $\mK$} if $\mF(u_{*})\le \mF(u)$ for all $u\in\mK$. 
A function $u_{*}\in\mK$ is called a \emph{local minimizer of \eqref{eq:main1} in $\mK$} if there exists $\delta>0$ such that $\mF(u_{*})\le \mF(u)$ for all $u\in\mK$ with $\norm{u-u_{*}}_{H^{1}(\mR^{d})}^{2} \le \delta$. 
\end{definition*}

We will show later in \Cref{thm:Lewy-Stampacchia} that every local minimizer $u_{*}$ satisfies the Euler-Lagrange equation 
\begin{equation}
\Delta u_{*} - Vu_{*}\chi_{\{u_{*}<1\}} = 0 \quad \text{a.e. in $\mR^{d}$.} \label{eq:EL} 
\end{equation} 
A similar variational problem was studied in the case $d=1$ in \cite{AKO24TwoStepMinimization2,AKO25TwoStepMinimization1}. They proved the existence of minimizer by using the Sobolev embedding $H^{1}(\mR^{1})\subset C^{0}(\mR^{1})$ and the fact that for each $u\in H^{1}(\mR^{1})$, there exists $a\in\mR^{1}$ such that $\abs{u(a)}=\norm{u}_{L^{\infty}(\mR)}$. These facts do not hold in higher dimensions, which prevents the arguments in \cite{AKO24TwoStepMinimization2, AKO25TwoStepMinimization1} from being generalized to the higher-dimensional case. 

To clarify the mathematical background of our problem while also providing a stronger biological motivation, we now explain in more detail how the problem arises. Following \cite{DD02PredatorPreyModel,DD03FBPopulationBiology,DD06UniquenessTheoremFB,DDM03Asymptotic,Du2002CompetitionModelI,Du2002CompetitionModelII,RT13IncresingPower}, we begin with the degenerate logistic model from population biology, which has been used to describe the effects of spatial heterogeneity in various population models: 
\begin{equation}
\Delta u -V u - bu^{p}=0, \label{eq:p-model}
\end{equation}
where $b(x)$ is a nonnegative function and $p>1$. Formally letting $p\rightarrow+\infty$ leads to  
\begin{equation*}
\Delta u - Vu=0 \text{ in $\{u<1\}$} ,\quad b(x)=0 \text{ for all $x\in\{u>1\}$.} 
\end{equation*}
Thus, if we choose $b>0$ in $\mR^{d}$, the above equation formally becomes 
\begin{equation*}
\Delta u  -Vu\chi_{\{u<1\}}=0, \quad u\le1,  
\end{equation*}
showing that our model can be viewed as a formal $p\rightarrow+\infty$ limit of the degenerate logistic model. 
We emphasize that the works cited above are formulated on bounded domains. 
In contrast, our work considers spatially heterogeneous environments in the whole space, described by a sign-changing potential $V$. In this setting, we focus on population localization, the formation of saturated regions, and their geometric structure. 
We note that the equation \eqref{eq:p-model} in $\mR^{d}$ was studied in \cite{DM01LogisticModelRn,DD03LogisticModelRn}. Thus, our result can be viewed as the formal limiting case $p\rightarrow+\infty$ of these works.

\subsection{Properties of local minimizers}  

As with other variational problems, we begin by studying the Euler-Lagrange equation associated with the minimization problem \eqref{eq:main1}, analogous to \cite[Lemma~2.2]{GS96FreeBoundaryPotential} and \cite[Proposition~4.2]{KSS23Minimization}. 

\begin{theorem}\label{thm:Lewy-Stampacchia} 
Let $V \in L^{\infty}(\mR^{d})$. If $u_{*}$ is a local minimizer of \eqref{eq:main1} in $\mK$, then $u_{*}\in W_{\rm loc}^{2,p}(\mR^{d})$ for all $1 \le p < \infty$, and it satisfies the Euler-Lagrange equation \eqref{eq:EL} and 
\begin{equation}
\Delta u_{*} = 0 \quad \text{ and} \quad V \le 0 \quad \text{a.e. in $\{u_{*}=1\}$}. \label{eq:overpopulated-region}
\end{equation}
\end{theorem}

\begin{remark}[Nonvanishing property of local minimizers]\label{rem:Landis} 
The set $\{u_{*}=0\}$ may be interpreted as an uninhabited region. Applying the unique continuation property for elliptic equations from sets of positive measure to \eqref{eq:EL} (see \cite{Regbaoiu2001MUCP} for $d\ge3$ and \cite{GL93MUCP} for $d=2$), we conclude that $\{u_{*}=0\}$ has Lebesgue measure zero whenever $u_{*}\not\equiv0$. 
\end{remark}

\begin{remark}[Saturated area]\label{rem:overpopulated} 
The set $\left\{u_{*}=1\right\}$ may indicate a saturated region. The inequality $V\le0$ in \eqref{eq:overpopulated-region} characterizes the phenomenon that saturated areas necessarily lie within the favorable region (i.e., the ``potential well''). Combining \eqref{eq:EL} with \eqref{eq:overpopulated-region} yields 
\begin{equation}
0\le \Delta u_{*}-V(x)u_{*} =V_{-}(x)\chi_{\{u_{*}=1\}}
\quad\text{a.e. in }\mR^{d}. \label{eq:LS-ineq}
\end{equation} 
Conversely, \eqref{eq:LS-ineq} implies both \eqref{eq:EL} and \eqref{eq:overpopulated-region}.  
\end{remark}

From \Cref{thm:Lewy-Stampacchia} we have  
\begin{equation} 
-\norm{V_{-}}_{L^{\infty}(\mR^{d})} \le \Delta u_{*} \le \norm{V_{+}}_{L^{\infty}(\mR^{d})} \quad \text{a.e. in $\mR^{d}$.} \label{eq:estimate-EL}
\end{equation}
Estimate \eqref{eq:estimate-EL} suggests that we prove the following proposition, which follows from a variant of the Harnack inequality (\Cref{lem:Harnack}): 

\begin{table}[H]
\rowcolors{2}{white}{lightgray!50}
\renewcommand{\arraystretch}{1.5} 
\centering
\begin{tabular}{c|c|c|c|c}
$v$ & $M_{1}$ & $M_{2}$ & $C_{1}$ & $C_{2}$ \\ 
\hline \hline 
$1-u_{*}$ & $\norm{V_{-}}_{L^{\infty}(\mR^{d})}$ & $\norm{V_{+}}_{L^{\infty}(\mR^{d})}$ & $\frac{2^{d}\norm{V_{-}}_{L^{\infty}(\mR^{d})} + \norm{V_{+}}_{L^{\infty}(\mR^{d})}}{2d}$ & $2^{d-1}\norm{V_{-}}_{L^{\infty}(\mR^{d})} + \frac{d\norm{V}_{L^{\infty}(\mR^{d})}}{d+1}$ \\ 
$u_{*}$ & $\norm{V_{+}}_{L^{\infty}(\mR^{d})}$ & $\norm{V_{-}}_{L^{\infty}(\mR^{d})}$ & $\frac{2^{d}\norm{V_{+}}_{L^{\infty}(\mR^{d})} + \norm{V_{-}}_{L^{\infty}(\mR^{d})}}{2d}$ & $2^{d-1}\norm{V_{+}}_{L^{\infty}(\mR^{d})} + \frac{d\norm{V}_{L^{\infty}(\mR^{d})}}{d+1}$
\end{tabular}
\caption{Choice of $(v,M_{1},M_{2})$ in \Cref{thm:Harnack-application} leading to \Cref{thm:continuity}.}
\label{tab:choice-parameters}
\end{table}

\begin{proposition}\label{thm:Harnack-application} 
If $v\in W_{\rm loc}^{2,1}(\mR^{d})$ satisfying 
\begin{equation*}
-M_{2} \le \Delta v \le M_{1} ,\quad v\ge 0 \quad \text{in $\mR^{d}$,}
\end{equation*}
for some nonnegative constants $M_{1}$ and $M_{2}$ and $\{v=0\}\neq\emptyset$, then 
\begin{enumerate}
\renewcommand{\labelenumi}{\theenumi}
\renewcommand{\theenumi}{\rm (\alph{enumi})} 
\item \label{itm:Harnack1} $v(x) \le C_{1}\sfd_{\{v=0\}}(x)^{2}$ for all $x\in\mR^{d}$, where $C_{1}=\frac{2^{d}M_{1}+M_{2}}{2d}$. 
\item \label{itm:Harnack2} $\abs{\nabla v(x)} \le C_{2}\sfd_{\{v=0\}}(x)$ for all $x\in\mR^{d}$, where $C_{2} =2^{d-1}M_{1} + \frac{d\max\{M_{1},M_{2}\}}{d+1}$. 
\end{enumerate} 
\end{proposition}

In view of \eqref{eq:estimate-EL}, setting $(v,M_{1},M_{2})$ as in \Cref{tab:choice-parameters} within \Cref{thm:Harnack-application} immediately yields the following theorem. 

\begin{theorem}\label{thm:continuity} 
Let $V \in L^{\infty}(\mR^{d})$ and let $u_{*}$ be a local minimizer of \eqref{eq:main1} in $\mK$. 
\begin{enumerate}
\renewcommand{\labelenumi}{\theenumi}
\renewcommand{\theenumi}{\rm (\alph{enumi})}
\item \label{itm:no-FB1} If $\left\{u_{*}=1\right\}\neq\emptyset$, then $0 \le 1-u_{*}(x) \le C_{1}\sfd_{\{u_{*}=1\}}(x)^{2}$ and $\abs{\nabla u_{*}(x)} \le C_{2}\sfd_{\{u_{*}=1\}}(x)$ for all $x\in\mR^{d}$. 
\item \label{itm:no-FB2} If $\left\{u_{*}=0\right\}\neq\emptyset$, then $0 \le u_{*}(x) \le C_{1}\sfd_{\{u_{*}=0\}}(x)^{2}$ and $\abs{\nabla u_{*}(x)} \le C_{2}\sfd_{\{u_{*}=0\}}(x)$ for all $x\in\mR^{d}$. 
\end{enumerate} 
with positive constants $C_{1}$ and $C_{2}$ given in \Cref{tab:choice-parameters}. 
\end{theorem} 

In \Cref{thm:Lewy-Stampacchia}, elliptic regularity implies that $u \in W^{2,p}_{\mathrm{loc}}(\mathbb{R}^d)$ for all $1 \le p < \infty$. 
By contrast, near the contact set, the sharper estimates derived in \Cref{thm:continuity}\ref{itm:no-FB1} allow us to upgrade the regularity of $u$ to the class $C^{1,1}$ across the free boundary $\partial\{u_{*}<1\}$ in the following sense: 
\begin{equation*}
\abs{\nabla u_{*}(x) - \nabla u_{*}(y)} = \abs{\nabla u_{*}(y)} \le C \sfd_{\{u_{*}=1\}}(x) \le C \abs{x-y} 
\end{equation*}
for all $x\in\{u_{*}=1\}$ and $y\in\mR^{d}$. Using \Cref{thm:continuity}\ref{itm:no-FB2}, one similarly obtains $C^{1,1}$ regularity across the free boundary $\partial\{u_{*}>0\}$. 

We also highlight a proposition that is useful for computing the associated energy. 

\begin{proposition}\label{rem:simple-obervations} 
Let $V \in L^{\infty}(\mR^{d})$ satisfy $V\ge \alpha^{2}$ in $\overline{B_{R}}^{\complement}$ for some $\alpha>0$ and $R>0$.  If $u_{*}$ is a local minimizer of \eqref{eq:main1} in $\mK$, then 
\begin{equation*}
\mF(u_{*}) = \int_{\{u_{*}=1\}} V(x)\,\rmd x. 
\end{equation*} 
\end{proposition}

\subsection{Existence and uniqueness of a global minimizer}  

Before stating our existence and uniqueness result, we introduce the \emph{fundamental tone} $\sigma_{*}(V)$ of the operator $-\Delta+V$, defined by: 
\begin{equation*}
\sigma_{*}(V) := \inf_{0\neq\phi\in C_{c}^{\infty}(\mR^{d})}\frac{\int_{\mR^{d}}(\abs{\nabla\phi}^{2}+V\phi^{2})\,\rmd x}{\int_{\mR^{d}}\phi^{2}\,\rmd x} \equiv \inf_{0\neq\phi\in C_{c}^{\infty}(\mR^{d})}\frac{\mF_{V}(\phi)}{\int_{\mR^{d}}\phi^{2}\,\rmd x} 
\end{equation*}
If $\sigma_{*}(V)>0$, then 
\begin{equation*}
\mF_{V}(u) \ge \sigma_{*}(V)\norm{u}_{L^{2}(\mR^{d})}^{2} > 0 \quad \text{for all $0\neq u\in H^{1}(\mR^{d})$,} 
\end{equation*}
which shows that $0$ is the unique minimizer of \eqref{eq:main1} in $\mK$.  

On the other hand, if $\sigma_{*}(V)<0$, then there exists $\phi_{0}\in C_{c}^{\infty}(\mR^{d})$ such that $\mF_{V}(\abs{\phi_{0}}) = \mF_{V}(\phi_{0}) < 0$. Consequently, 
\begin{equation*}
\mF_{V}\left(\frac{\abs{\phi_{0}}}{\norm{\phi_{0}}_{L^{\infty}(\mR^{d})}}\right) < 0 ,\quad \frac{\abs{\phi_{0}}}{\norm{\phi_{0}}_{L^{\infty}(\mR^{d})}}\in\mK, 
\end{equation*}
and hence 
\begin{equation}
m_{*}:=\inf_{u\in\mK}\mF(u) < 0. \label{eq:nontrivial}
\end{equation} 
We also remark that, if $V\ge0$ a.e. outside some ball $B_{R}$, then $\mF|_{\mK}$ is bounded from below: 
\begin{equation*}
\mF(u) \ge \int_{B_{R}} V(x)\abs{u}^{2}\,\rmd x \ge - \abs{B_{R}} \norm{V_{-}}_{L^{\infty}(\mR^{d})} \quad \text{for all $u\in\mK$.} 
\end{equation*}
These observations motivate the following assumption on $V$: 

\begin{assumption}\label{assu:1} 
Suppose that $V\in L^{\infty}(\mR^{d})$ satisfies $\sigma_{*}(V)<0$ and 
\begin{equation*}
V(x) \ge \alpha^{2} \text{ for a.e. $x\in B_{R}^{\complement}$} 
\end{equation*}
for some $\alpha>0$ and $R>0$. 
\end{assumption} 

\Cref{assu:1} may be interpreted, for instance, as migration toward a green oasis in an otherwise barren desert. The following example provides an easily verifiable sufficient condition for $\sigma_{*}(V)<0$. 

\begin{example*}[``potential well''] 
Suppose there exists a nonempty bounded domain $\emptyset\neq G\subset B_{R}$ such that 
\begin{equation*}
\sup_{x\in G}V(x) < -\lambda^{*}(G), \quad \lambda^{*}(G) := \inf_{0\neq\psi\in  C_{c}^{\infty}(G)} \frac{\norm{\nabla\psi}_{L^{2}(G)}^{2}}{\norm{\psi}_{L^{2}(G)}^{2}}. 
\end{equation*} 
Given any $0<\epsilon<-\sup_{x\in G}V(x)-\lambda^{*}(G)$, there exists a nonzero function $0\neq\psi_{0}\in C_{c}^{\infty}(G)$ such that 
\begin{equation*}
\norm{\nabla\psi_{0}}_{L^{2}(\mR^{d})}^{2} < (\lambda^{*}(G)+\epsilon)\norm{\psi_{0}}_{L^{2}(\mR^{d})}^{2}. 
\end{equation*} 
Consequently, 
\begin{equation*}
\norm{\nabla\psi_{1}}_{L^{2}(\mR^{d})}^{2} < (\lambda^{*}(G)+\epsilon)\norm{\psi_{1}}_{L^{2}(\mR^{d})}^{2} ,\quad \psi_{1}:=\frac{\abs{\psi_{0}}}{\norm{\psi_{0}}_{L^{\infty}(\mR^{d})}}\chi_{G} \in \mK. 
\end{equation*}
It follows that 
\begin{equation*}
\mF_{V}(\psi_{1}) \le \int_{G}(\lambda^{*}(G)+\epsilon+V)\psi_{1}^{2}\,\rmd x \le \overbrace{(\lambda^{*}(G)+\epsilon+\sup_{x\in G}V(x))}^{<0}\norm{\psi_{1}}_{L^{2}(G)}^{2}
\end{equation*}
and hence $\sigma_{*}(V)<0$. 
The quantity $\lambda^{*}(G)$ is referred to as the \emph{fundamental tone} of $G$, see also \cite[Proposition~2.6]{GS24PartialBalayageHelmholtz} for equivalent characterizations\footnote{In \cite[Proposition~2.8]{GS24PartialBalayageHelmholtz}, necessary and sufficient conditions for the validity of the maximum principle for the Helmholtz operator are established.}. 
It is well known that $\lambda^{*}(G) > 0$ and coincides with the first Dirichlet eigenvalue of $-\Delta$ on $G$. Moreover, $\lambda^{*}(G_1) \ge \lambda^{*}(G_2)$ whenever $G_1 \subset G_2$, and $\lambda^{*}(B_r) = r^{-2}\lambda^{*}(B_1)$ for all $r > 0$.
\end{example*}

Our next main result establishes the existence and uniqueness of a global minimizer, as stated below, see \Cref{sec:numerical} for numerical examples: 

\begin{theorem}\label{thm:well-posedness} 
We assume that $V$ satisfies \Cref{assu:1}. If there exists a bounded open set $G$ such that $G={\rm int}\,\overline{G}$, $\overline{G}^{\complement}$ is connected and satisfies the condition of the interior sphere, and  
\begin{equation}
V<0 \text{ a.e. in $G$} \quad \text{and} \quad V\ge0 \text{ a.e. in $G^{\complement}$,} \label{eq:boundary-assumption-V-minus}
\end{equation}
then there exists a unique global minimizer $u_{*}$ of \eqref{eq:main1} in $\mK$. The contact set $\{u_*=1\}$ of the minimizer has positive measure and is contained in $G$. 
Let $\alpha_{0}:= \norm{V_{+}}_{L^{\infty}(\mR^{d})}^{1/2}$. There exist positive constants $c=c(d,V)$, $C=C(R,d,\alpha)$ and $C'=C'(R,d,\alpha,\alpha_{0},\norm{V_{-}}_{L^{\infty}(\mR^{d})})$ such that 
\begin{equation}
\left\{\begin{aligned}
& u_{*}(x) \ge c\abs{x}^{-\frac{d-1}{2}} e^{-\alpha_{0}\abs{x}} &&\text{for all $x\in B_{2R}^{\complement}$,} \\ 
& u_{*}(x) \le C\abs{x}^{-\frac{d-1}{2}} e^{-\alpha\abs{x}} && \text{for all $x\in \mR^{d}\setminus\{0\}$,}
\end{aligned}\right.  \label{eq:decay}
\end{equation}
and 
\begin{equation}
\abs{\nabla u_{*}(x)} \le  C'\min\left\{1,\abs{x}^{-\frac{d-1}{2}}e^{-\alpha \abs{x}}\right\} \quad \text{for all $x\in\mR^{d}$.} \label{eq:decay-gradient}
\end{equation} 
If, in addition, $V$ is radially symmetric, then the unique minimizer is also radially symmetric. 
\end{theorem}

\begin{remark*}
The estimate \eqref{eq:decay} remains valid for local minimizers $u_{*}$, but the constant $c$ depends on $u_{*}$. 
\end{remark*}

\begin{remark*}
The uniqueness of the minimizer follows from \eqref{eq:boundary-assumption-V-minus}. The proof relies on the unique continuation property for elliptic equations from sets of positive measure (see \cite{Regbaoiu2001MUCP} for $d\ge3$ and \cite{GL93MUCP} for $d=2$). For general $V$, see \Cref{lem:bound} for properties of arbitrary global minimizers. 
A key step in the proof is to show that $\{\overline{u}_{*}=1\}\subset G$, which in turn implies that $\{\underline{u}_{*}<1\}\cap\{\overline{u}_{*}=1\}$ has Lebesgue measure zero, where $\overline{u}_{*}$ and $\underline{u}_{*}$ denote the largest and smallest minimizers, respectively (see \Cref{lem:bound} for their precise definitions). Consequently, the result in \Cref{thm:well-posedness} remains valid if the assumption \eqref{eq:boundary-assumption-V-minus} is replaced by the assumption that ${V=0}$ has Lebesgue measure zero.
\end{remark*}

\Cref{assu:1} ensures that all global minimizers of \eqref{eq:main1} are nontrivial, indicating that the favorable environment $G$ attracts the species from afar.
Since the embedding $H^{1}(\mR^{d})\subset L^{2}(\mR^{d})$ is not compact, several standard compactness theorems are not applicable. As a result, establishing the existence of a global minimizer of \eqref{eq:main1} in $\mK$ is not straightforward. 
As a direct consequence of \eqref{eq:decay}, we know that the set $\{u_{*}=1\}\neq\emptyset$ is bounded. This immediately establishes the existence of a solution to the following inverse shape determination problem: finding a bounded domain $D \subset \mR^{d}$ and $u \in H^{1}(D^{\complement})$ such that 
\begin{equation*}
\left\{\begin{aligned}
& \Delta u - V(x)u = 0 && \text{in $D^{\complement}$,} \\ 
& u = 1 ,\quad \abs{\nabla u}=0 && \text{on $\partial D$.} 
\end{aligned}\right. 
\end{equation*}

We can also slightly strengthen the result when the potential $V$ is radially symmetric. 

\begin{theorem}\label{thm:radially-symmetric} 
We assume that $V$ satisfies \Cref{assu:1} and is radially symmetric. If $V$ satisfies \eqref{eq:boundary-assumption-V-minus} with $G=B_{R_{0}}$ for some $R_{0}\in(0,R]$, then there exists $R_*\in (0,R_{0})$ such that $\{u_*=1\}=\overline{B_{R_*}}$ and $\partial_{r}u_{*}<0$ in $\overline{B_{R_{*}}}^{\complement}$, where $r=\abs{x}$. 
\end{theorem}

Analytical examples are provided in \Cref{sec:analytic-examples}. 
As illustrated in \Cref{fig:radially-symmetric-minimizer}, the numerical results indicate that the saturated region $\{u_{*}=1\}$ expands when the potential well becomes deeper (increasing $\beta$ or decreasing $\alpha$) or wider (increasing $G$).

\subsection{Positivity of global minimizers}

The following theorem establishes the positivity of the global minimizers described in \Cref{thm:well-posedness} by Hopf's maximum principle (\Cref{cor:Hopf}), under additional assumptions on $V$: 

\begin{theorem}\label{thm:positivity}
We assume that $V$ satisfies \Cref{assu:1}, and let $u_{*}$ be any global minimizer of \eqref{eq:main1} in $\mK$. If there exists a bounded connected open set $G$ that satisfies the condition of the interior sphere, $\overline{G}^{\complement}$ is connected, and 
\begin{equation*}
V \le 0 \text{ a.e. in $G$} \quad \text{and} \quad V \ge 0 \text{ a.e. in $G^{\complement}$,}
\end{equation*} 
then $u_{*}(x)>0$ for all $x\in\mR^{d}$. 
\end{theorem}

\begin{remark*}
The interior sphere condition on $G$ is used only to apply Hopf's lemma (\Cref{cor:Hopf}) on $\partial G$ and establish that $u_{*}>0$ on $\partial G$. 
\end{remark*}

\subsection{Geometry of the free boundary \texorpdfstring{$\partial\{u_{*} < 1\}$}{}}

A natural question that arises is whether free boundary methods can be employed to analyze the $C^{1}$ regularity of the free boundary $\partial\{u_{*} < 1\}$. Recently, Mikko Salo and Henrik Shahgholian developed new free boundary methods in \cite{SS25vanishingcontrast}, which have subsequently been applied to scattering problems for acoustic and elastic waves, see \cite{KSS26ElasticCorner,KSS26AnisotropicII}. By using their methods, we are able to prove the following result for any local minimizers: 

\begin{theorem}\label{thm:FB}
Suppose that all assumptions in \Cref{thm:continuity} hold. If $D:=\{u_{*}<1\}$ is a Lipschitz domain, $V$ is $C^{\alpha}$ for some $0<\alpha<1$ near $x_{0} \in \partial D$ and $V(x_{0})<0$,
then $u_{*}$ reaches the optimal regularity $C^{1,1}$ near $x_{0}$ and the following holds true: 
\begin{enumerate}
\renewcommand{\labelenumi}{\theenumi}
\renewcommand{\theenumi}{\rm (\alph{enumi})}
\item If $d=2$ and $\partial D$ is piecewise $C^{1}$, then $\partial D$ is $C^{1}$ near $x_{0}$. 
\item If $d\ge 3$, then $x_{0}$ is not an edge point described in \cite[Definition~1.3]{SS25vanishingcontrast}. 
\end{enumerate}
\end{theorem}

In view of \eqref{eq:EL}, \Cref{thm:FB} follows directly from \Cref{prop:FB} by taking $D=\{u_{*}<1\}$ and $f=-Vu_{*}$. 

Setting $w_{*}:=1-u_{*}\ge0$ in $\mR^{d}$, the Euler-Lagrange equation \eqref{eq:EL} becomes  
\begin{equation}
\Delta w_{*}  = f\chi_{\{w_{*}>0\}} \quad f:=-V(1-w_{*}) \quad \text{a.e. in $\mR^{d}$.} \label{eq:FB-classical}
\end{equation}
Thus, the problem is equivalently one of studying the free boundary $\partial\{w_{*}>0\}$. Suppose that all assumptions in \Cref{thm:well-posedness} hold. Then 
\begin{equation*}
\partial\{w_{*}>0\} = \partial\{w_{*}=0\} \subset \{w_{*}=0\} \subset G. 
\end{equation*} 
Recall that $G$ is an open set. If, in addition, $V\in C(G)$ and $V<0$ in $G$, then $f\in C(G)$ and it is strictly positive on $\partial\{w_{*}>0\}$, which shows that \eqref{eq:FB-classical} is precisely the classical \emph{no-sign obstacle problem}, see, for example, \cite[(1.4) in Section 1.1.2 and Problem~A in Section~1.2.1]{PSU12FreeBoundary} and \cite[(1.12)]{SS21NonscatteringFreeBoundary}. This problem has been extensively studied, and we refer the reader to \cite{ALS13Regularity} and references therein. We do not attempt to provide an exhaustive account here. We mention, however, that the classical result for the special case in which $\partial\{w_{*}>0\}$ is known a priori to be Lipschitz was established by Caffarelli \cite{Caf77RegularityFreeBoundary}. 

The following theorem follows immediately from the results of \cite{ALS13Regularity}, which, roughly speaking, show that if $\partial\{u_{*}=1\}$ is sufficiently regular (e.g., $\partial\{u_{*}=1\}$ is Lipschitz) near $x_{0}$, then it is necessarily $C^{1}$ near $x_{0}$: 

\begin{theorem}
Suppose that all assumptions in \Cref{thm:well-posedness} hold. If, in addition, $V\in C^{\rm Dini}(G)$ and $V<0$ in $G$, then the unique global minimizer $u_{*}$ of \eqref{eq:main1} in $\mK$ is $C^{1,1}$ near $\partial\{u_{*}=1\}$, and for each $x_{0}\in\partial\{u_{*}=1\}$, there exist a modulus of continuity $\sigma$ and $r_{0}>0$ such that, whenever  
\begin{equation*}
\frac{{\rm MD}(\{u_{*}=1\}\cap B_{r}(x_{0}))}{r} > \sigma(r) 
\end{equation*}
for some $r<r_{0}$, there exists $\rho>0$ such that $\partial(\{u_{*}<1\}) \cap B_{\rho}(x_{0})$ is a $C^{1}$-graph. Here ${\rm MD}$ denotes the minimal diameter. 
\end{theorem}

\addtocontents{toc}{\SkipTocEntry}
\subsection*{Organization}  

In \Cref{sec:proof-thm:continuity}, we prove \Cref{thm:Lewy-Stampacchia} and \Cref{thm:Harnack-application}. In \Cref{sec:minimizers}, we establish \Cref{thm:well-posedness} and then investigate the properties of global minimizers by proving \Cref{thm:positivity}. Next, in \Cref{sec:analytic-examples}, we prove \Cref{thm:radially-symmetric} and present analytic examples of radially symmetric solutions. Finally, \Cref{sec:numerical} presents several numerical examples. 
We also recall a comparison principle and several Harnack inequalities in \Cref{sec:harnack}, some geometric results on free boundaries in \Cref{sec:FB}, and basic properties of Bessel functions in \Cref{sec:Bessel}. 

\section{Properties of local minimizers\label{sec:proof-thm:continuity}} 

We begin with the proof of \Cref{thm:Lewy-Stampacchia}. 

\begin{proof}[Proof of \Cref{thm:Lewy-Stampacchia}]
We first prove that 
\begin{equation}
\Delta u_{*} - V(x) u_{*}\ge0 \quad \text{in distribution sense.} \label{eq:LS-lower-bound}
\end{equation}
Let $\phi \in C_{c}^{\infty}(\mR^{d})$ be a nonnegative test function. For each $\epsilon>0$, 
we define $v_{\epsilon}:=\max\{u_{*}-\epsilon\phi,0\}$. 
By \cite[Theorem~A.1]{KS00IntroductionVariationalInequalities}, 
we have $v_\epsilon \in H^1(\mathbb{R}^d)$. 
Moreover, since 
$0 \le v_{\epsilon} \le u_{*}\le 1$ in $\mR^{d}$, 
it follows that $v_\epsilon \in \mK$ for all $\epsilon > 0$. 
In addition, the same theorem yields the following properties:
\begin{equation}
(v_{\epsilon},\nabla v_{\epsilon})
= 
\left\{\begin{aligned}
& (0,0) && \text{a.e. on $\{v_{\epsilon}=0\}$,} \\ 
& {\displaystyle \big(u_{*}-\epsilon\phi, \nabla (u_{*} - \epsilon\phi)\big)} && \text{a.e. on $\{v_{\epsilon} >0\}$.}
\end{aligned}\right.  \label{veps}
\end{equation}
We remark that, for $\tilde\epsilon >\epsilon >0$,
\begin{align}
&\{v_{\tilde\epsilon} =0\}
\supset
\{v_\epsilon =0\}
,\quad
\bigcap_{\epsilon >0}\{v_\epsilon =0\}=\{u_*=0\},
\label{eps-set-v1}\\
&\{v_{\tilde\epsilon} >0\}
\subset
\{v_\epsilon >0\}
,\quad
\bigcup_{\epsilon >0}\{v_\epsilon >0\}=\{u_*>0\}.
\label{eps-set-v2}
\end{align}
Note that $v_{\epsilon}\rightarrow u_{*}$ in $H^{1}(\mR^{d})$ as $\epsilon\rightarrow 0_{+}$. 
Indeed, using \eqref{eps-set-v1}, this convergence follows from the estimate 
\begin{align*}
\norm{v_{\epsilon}-u_{*}}_{H^{1}(\mR^{d})}^{2} 
&= \int_{\{v_{\epsilon}=0\}} \left( \abs{\nabla u_{*}}^{2} + \abs{u_{*}}^{2} \right) \, \rmd x + \epsilon^{2}\int_{\{v_{\epsilon}>0\}} \left( \abs{\nabla\phi}^{2} + \abs{\phi}^{2} \right) \, \rmd x\\
&\le \int_{\{v_{\epsilon}=0\}} \left( \abs{\nabla u_{*}}^{2} + 
\abs{u_{*}}^{2} \right) \, \rmd x + \epsilon^{2}
\norm{\phi}_{H^{1}(\mR^{d})}^{2}\\
&\longrightarrow \int_{\{u_{*}=0\}} \left( \abs{\nabla u_{*}}^{2} + 
\abs{u_{*}}^{2} \right) \, \rmd x =0
\quad\mbox{as}~\epsilon \to 0_+.
\end{align*} 
Since $u_{*}$ is a local minimizer of \eqref{eq:main1}, for sufficiently small $\epsilon>0$, we have
\begin{align}
0 &\le \mF(v_{\epsilon}) - \mF(u_{*})  \notag \\
 &= -2\epsilon \int_{\{v_{\epsilon}>0\}} \left( \nabla u_{*}\cdot\nabla \phi + V(x)u_{*}\phi \right) \, \rmd x \notag \\
& 
\quad + \epsilon^{2}\int_{\{v_{\epsilon}>0\}}\left(\abs{\nabla\phi}^{2}+V(x)\abs{\phi}^{2}\right)\,\rmd x  
- \int_{\{v_{\epsilon}=0\}} \left( \abs{\nabla u_{*}}^{2} + V(x)\abs{u_{*}}^{2} \right) \,\rmd x \notag \\
&\le 
-2\epsilon \int_{\{v_{\epsilon}>0\}} \left( \nabla u_{*}\cdot\nabla \phi + V(x)u_{*}\phi \right) \, \rmd x \notag \\
& 
\quad + \epsilon^{2}\int_{\{v_{\epsilon}>0\}}\left(\abs{\nabla\phi}^{2}+V(x)\abs{\phi}^{2}\right)\,\rmd x  
+\epsilon^{2}\norm{V_-}_{L^{\infty}(\mR^{d})}\int_{\{v_{\epsilon}=0\}}\abs{\phi}^{2}\,\rmd x,
 \label{eps-ineq}
\end{align}
where, since $u_{*}\le\epsilon\phi$ in $\{v_{\epsilon}=0\}$, we have used the following inequality:
\begin{equation*}
- \int_{\{v_{\epsilon}=0\}} \left( \abs{\nabla u_{*}}^{2} + V(x)\abs{u_{*}}^{2} \right) \,\rmd x\le 
\int_{\{v_{\epsilon}=0\}}V_-(x)\abs{u_{*}}^{2}\,\rmd x \le
\epsilon^{2}\norm{V_-}_{L^{\infty}(\mR^{d})}\int_{\{v_{\epsilon}=0\}}\abs{\phi}^{2}\,\rmd x. 
\end{equation*}
Dividing the inequality \eqref{eps-ineq} by $2\epsilon$ and letting $\epsilon\rightarrow 0_{+}$, 
in view of \eqref{eps-set-v2}, the Lebesgue dominated convergence theorem ensures that 
\begin{equation*}
\begin{aligned} 
0 &\le -\int_{\{u_{*}>0\}} \left( \nabla u_{*}\cdot\nabla \phi + V(x)u_{*}\phi \right) \,\rmd x \\
& = -\int_{\mR^{d}} \left( \nabla u_{*}\cdot\nabla \phi + V(x)u_{*}\phi \right) \,\rmd x = \,_{\mD'(\mR^{d})}\br{\Delta u_{*} - Vu_{*} , \phi}_{\mD(\mR^{d})}, 
\end{aligned}
\end{equation*}
which conclude \eqref{eq:LS-lower-bound}. 

We now prove that 
\begin{equation}
\Delta u_{*} - V(x) u_{*} \le V_{-}(x)\chi_{\{u_{*}=1\}} \quad \text{in distribution sense.} \label{eq:LS-upper-bound}
\end{equation}
Let $\phi\in C_{c}^{\infty}(\mR^{d})$ be a nonnegative test function. For each $\epsilon>0$, we define  
$w_{\epsilon}:=\min\{u_{*}+\epsilon\phi,1\}\in H^1_{\rm loc}(\mR^d)$.
Similarly to \eqref{veps},
we have
\begin{equation}
(w_{\epsilon},\nabla w_{\epsilon})
= 
\left\{\begin{aligned}
& (1,0) && \text{a.e. on $\{w_{\epsilon}=1\}$,} \\ 
& {\displaystyle \big(u_{*}+\epsilon\phi, \nabla (u_{*} +\epsilon\phi)\big)} && \text{a.e. on $\{w_{\epsilon} <1\}$.}
\end{aligned}\right.  \label{weps}
\end{equation}
Since 
$0\le u_* \le w_{\epsilon} \le u_{*}+\epsilon \phi$ in $\mR^{d}$, 
it follows that $w_\epsilon \in \mK$ for all $\epsilon > 0$. 
We remark that, for $\tilde\epsilon >\epsilon >0$,
\begin{align}
&\{w_{\tilde\epsilon} =1\}
\supset
\{w_\epsilon =1\}
,\quad
\bigcap_{\epsilon >0}\{w_\epsilon =1\}=\{u_*=1\},
\label{eps-set-w1}\\
&\{w_{\tilde\epsilon} <1\}
\subset
\{w_\epsilon <1\}
,\quad
\bigcup_{\epsilon >0}\{w_\epsilon <1\}=\{u_*<1\}.
\label{eps-set-w2}
\end{align}
Note that $w_{\epsilon}\rightarrow u_{*}$ in $H^{1}(\mR^{d})$ as $\epsilon\rightarrow 0_{+}$.
Indeed, using \eqref{eps-set-w1}, this convergence follows from the estimate 
\begin{align*}
\norm{w_{\epsilon}-u_{*}}_{H^{1}
(\mR^{d})}^{2} 
&= 
\int_{\{w_\epsilon=1\}} (\abs{\nabla u_{*}}^{2} + \abs{1-u_{*}}^{2}) \,\rmd x + \epsilon^{2} \int_{\{w_\epsilon <1\}} (\abs{\nabla\phi}^{2}+\abs{\phi}^{2})\,\rmd x\\
&\le \int_{\{w_{\epsilon}=1\}} \left( \abs{\nabla u_{*}}^{2} + 
\abs{1-u_{*}}^{2} \right) \, \rmd x + \epsilon^{2}
\norm{\phi}_{H^{1}(\mR^{d})}^{2}\\
&\longrightarrow \int_{\{u_{*}=1\}} \left( \abs{\nabla u_{*}}^{2} + 
\abs{1-u_{*}}^{2} \right) \, \rmd x =0
\quad\mbox{as}~\epsilon \to 0_+.
\end{align*}
Since $u_{*}$ is a local minimizer of \eqref{eq:main1}, for sufficiently small $\epsilon>0$, we have
\begin{align}
0 &\le \mF(w_{\epsilon}) - \mF(u_{*})  \notag \\
 &= 2\epsilon \int_{\{w_{\epsilon}<1\}} \left( \nabla u_{*}\cdot\nabla \phi + V(x)u_{*}\phi \right) \, \rmd x \notag \\
& 
\quad + \epsilon^{2}\int_{\{w_{\epsilon}<1\}}\left(\abs{\nabla\phi}^{2}+V(x)\abs{\phi}^{2}\right)\,\rmd x  
- \int_{\{w_{\epsilon}=1\}} \left( \abs{\nabla u_{*}}^{2} + V(x)(\abs{u_{*}}^{2}-1) \right) \,\rmd x \notag \\
&\le 
2\epsilon \int_{\{w_{\epsilon}<1\}} \left( \nabla u_{*}\cdot\nabla \phi + V(x)u_{*}\phi \right) \, \rmd x \notag \\
& 
\quad + \epsilon^{2}\int_{\{w_{\epsilon}<1\}}\left(\abs{\nabla\phi}^{2}+V(x)\abs{\phi}^{2}\right)\,\rmd x  
+2\epsilon\int_{\{w_{\epsilon}=1\}}V_+(x)\phi\,\rmd x,
\label{eps-ineq-w}
\end{align}
where, since $1-u_{*}\le\epsilon\phi$ in $\{w_{\epsilon}=1\}$, we have used the following inequality:
\begin{align*}
- \int_{\{w_{\epsilon}=1\}} \left( \abs{\nabla u_{*}}^{2} + V(x)(\abs{u_{*}}^{2}-1) \right) \,\rmd x
&\le 
\int_{\{w_{\epsilon}=1\}}V(x)
(1-u_{*})(1+u_{*})\,\rmd x\\
&\le
2\epsilon\int_{\{w_{\epsilon}=1\}}V_+(x)\phi\,\rmd x. 
\end{align*}
Dividing the inequality \eqref{eps-ineq-w} by $2\epsilon$ and letting $\epsilon\rightarrow 0_{+}$, 
in view of \eqref{eps-set-w2}, the Lebesgue dominated convergence theorem ensures that 
\begin{equation*}
\begin{aligned} 
0 &\le \int_{\{u_{*}<1\}} \left( \nabla u_{*}\cdot\nabla \phi + V(x)u_{*}\phi \right) \,\rmd x 
+\int_{\{u_{*}=1\}} V_+(x)\phi \,\rmd x \\
& = \int_{\mR^{d}} \left( \nabla u_{*}\cdot\nabla \phi + V(x)u_{*}\phi \right) \,\rmd x 
+\int_{\{u_{*}=1\}} V_-(x)\phi \,\rmd x \\
&= \,_{\mD'(\mR^{d})}\br{-\Delta u_{*} +Vu_{*}+V_-\,\chi_{\{u_*=1\}} , \,\phi}_{\mD(\mR^{d})}, 
\end{aligned}
\end{equation*}
which conclude \eqref{eq:LS-upper-bound}. 

Since $V_{-}\chi_{\{u_{*}=1\}} \in L^{\infty}(\mR^{d})$, it follows from \eqref{eq:LS-lower-bound} and \eqref{eq:LS-upper-bound} that $(\Delta - V) u_{*} \in L^{\infty}(\mR^{d})$. By elliptic regularity, this yields $u_{*} \in W_{\rm loc}^{2,p}$ for all $1\le p<\infty$ and 
\begin{equation}
0 \le \Delta u_{*} - V(x) u_{*} \le V_{-}(x)\chi_{\{u_{*}=1\}} \quad \text{a.e. in $\mR^{d}$}. \label{eq:PDE-minimizer-1}
\end{equation} 
Applying \cite[Lemma~7.7]{GT01Elliptic} to $\nabla u_{*}$ yields
\begin{equation} 
\partial_{i}\partial_{j} u_{*}=0 \quad \text{a.e. in $\{u_{*}=1\}$} \label{eq:EL2}
\end{equation} 
for all $i,j\in\{1,\cdots,d\}$. Combining \eqref{eq:EL2} with the lower bound in \eqref{eq:PDE-minimizer-1} yields \eqref{eq:overpopulated-region}. Finally, combining \eqref{eq:PDE-minimizer-1} and \eqref{eq:overpopulated-region} yields the Euler-Lagrange equation \eqref{eq:EL}. 
\end{proof}

We next present the proof of \Cref{thm:Harnack-application}. 

\begin{proof}[Proof of \Cref{thm:Harnack-application}]
It suffices to show \Cref{thm:Harnack-application} for all $x_{0}\notin\{v=0\}$. 
For each $r>0$, we first apply a variant of the Harnack inequality (\Cref{lem:Harnack}\ref{itm:Harnack-b}) to obtain 
\begin{equation}
v(y+x_{0}) \le r^{d} \frac{r+\abs{y}}{(r-\abs{y})^{d-1}}\left(\frac{1}{r^{2}} \dashint_{\partial B_{r}(x_{0})} v \,\rmd S + \frac{M_{2}}{2d} \right) \quad \text{for all $y\in B_{r}(0)$.} \label{eq:Harnack-apply1} 
\end{equation}
Setting $y=0$ in \eqref{eq:Harnack-apply1} yields 
\begin{equation*}
v(x_{0}) \le  \dashint_{\partial B_{r}(x_{0})} v \,\rmd S + r^{2} \frac{M_{2}}{2d}  \quad \text{for all $r>0$,} 
\end{equation*}
that is, 
\begin{equation}
\frac{1}{r^{2}} \dashint_{\partial B_{r}(x_{0})} v \,\rmd S \ge \frac{1}{r^{2}}v(x_{0}) - \frac{M_{2}}{2d} \quad \text{for all $r>0$.} \label{eq:Harnack-apply2} 
\end{equation}
Let $\epsilon>0$ and we apply another variant of the Harnack inequality (\Cref{lem:Harnack}\ref{itm:Harnack-a}) with $r=\sfd_{\{v=0\}}(x_{0})+\epsilon$ to see that 
\begin{equation}
\begin{aligned}
v(x) &\ge (\sfd_{\{v=0\}}(x_{0})+\epsilon)^{d} \frac{\sfd_{\{v=0\}}(x_{0})+\epsilon-\abs{x}}{(\sfd_{\{v=0\}}(x_{0})+\epsilon+\abs{x})^{d-1}} \times \\ 
& \quad \times \left(\frac{1}{(\sfd_{\{v=0\}}(x_{0})+\epsilon)^{2}} \dashint_{\partial B_{\sfd_{\{v=0\}}(x_{0})+\epsilon}(x_{0})} v \,\rmd S - \frac{2^{d-1}M_{1}}{d} \right)
\end{aligned} \label{eq:important-obs-Harnack}
\end{equation}
for all $x\in B_{\sfd_{\{v=0\}}(x_{0})+\epsilon}(x_{0})$. Since $B_{\sfd(x_{0})}(x_{0})$ touches the boundary $\partial\{v=0\}$, then there exists $x'\in B_{\sfd_{\{v=0\}}(x_{0})+\epsilon}(x_{0})$ such that $v(x')=0$. Substituting $x=x'$ into \eqref{eq:important-obs-Harnack}, and then taking the limit $\epsilon\rightarrow 0$ gives 
\begin{equation}
\frac{1}{\sfd_{\{v=0\}}(x_{0})^{2}} \dashint_{\partial B_{\sfd_{\{v=0\}}(x_{0})}(x_{0})} v \,\rmd S \le \frac{2^{d-1}M_{1}}{d} \label{eq:important-obs-Harnack-limit} 
\end{equation}
We now combine \eqref{eq:Harnack-apply2} (with $r=\sfd_{\{v=0\}}(x_{0})$) and \eqref{eq:important-obs-Harnack-limit} to deduce 
\begin{equation*}
\frac{2^{d-1}M_{1}}{d} \ge \frac{1}{\sfd_{\{v=0\}}(x_{0})^{2}}v(x_{0}) - \frac{M_{2}}{2d}, 
\end{equation*}
which concludes \Cref{thm:Harnack-application}\ref{itm:Harnack1}.

On the other hand, we use another variant of Harnack inequality (\Cref{lem:Harnack}) with $M_{3}=\max\{M_{1},M_{2}\}$ and $r=\sfd_{\{v=0\}}(x_{0})$ to see that 
\begin{equation}
\abs{\nabla v(x_{0})} \le d \left(\frac{1}{\sfd_{\{v=0\}}(x_{0})}\dashint_{\partial B_{\sfd_{\{v=0\}}(x_{0})}(x_{0})} v \,\rmd S + \frac{\max\{M_{1},M_{2}\}}{d+1}\sfd_{\{v=0\}}(x_{0})\right)  \label{eq:Harnack-apply3}
\end{equation}
Finally, combining \eqref{eq:important-obs-Harnack-limit} with \eqref{eq:Harnack-apply3}, we conclude \Cref{thm:Harnack-application}\ref{itm:Harnack2}.
\end{proof}

Finally, we conclude this section by proving \Cref{rem:simple-obervations}, which records a useful property of local minimizers. 

\begin{proof}[Proof of \Cref{rem:simple-obervations}]
By the assumption $V\ge \alpha^{2}$ in $\overline{B_{R}}^{\complement}$ and the Lax-Milgram theorem \cite[Corollary~5.8]{Bre11PDE}, it follows that $u_{*}\in H^{1}(\overline{B_{R}}^{\complement})$ is the unique minimizer of the following minimizing problem: 
\begin{equation*}
\begin{aligned}
& \text{minimizing} && \frac{1}{2}\int_{\overline{B_{R}}^{\complement}}(\abs{\nabla v}^{2} + Vv^{2})\,\rmd x \\ 
& \text{subject to} && \left\{ v\in H^{1}(\overline{B_{R}}^{\complement}) : v|_{\partial B_{R}} = u_{*}\right\}, 
\end{aligned}
\end{equation*}
which can be characterized by 
\begin{equation*}
\int_{\overline{B_{R}}^{\complement}} (\nabla u_{*}\cdot\nabla\varphi + Vu_{*}\varphi)\,\rmd x = 0 \text{ for all $\varphi\in H_{0}^{1}(\overline{B_{R}}^{\complement})$.}  
\end{equation*}
Let $\phi\in C_{c}^{\infty}(\mR^{d})$ satisfy $0\le\phi\le 1$ in $\mR^{d}$ and $\phi=1$ in a neighborhood of $\overline{B_{R}}$. Choosing 
\begin{equation*}
\varphi = u_{*} - u_{*}\phi \in H_{0}^{1}(\overline{B_{R}}^{\complement}), 
\end{equation*}
we reach 
\begin{equation*}
\begin{aligned}
\int_{\overline{B_{R}}^{\complement}} (\abs{\nabla u_{*}}^{2} + Vu_{*}^{2})\,\rmd x &= \int_{\overline{B_{R}}^{\complement}} (\nabla u_{*}\cdot\nabla (u_{*}\phi) + Vu_{*}^{2}\phi)\,\rmd x \\ 
& = -\int_{\partial B_{R}} u_{*}\partial_{r}u_{*} \,\rmd S_{x} - \int_{\overline{B_{R}}^{\complement}} (\Delta u_{*}-Vu_{*}) u_{*}\phi\,\rmd x \\ 
& = -\int_{\partial B_{R}} u_{*}\partial_{r}u_{*} \,\rmd S_{x},
\end{aligned}
\end{equation*}
where $r=|x|$. 
Since $\{u_{*}=1\}\subset B_{R}$, consequently by \eqref{eq:EL}, 
\begin{equation*}
\begin{aligned}
\mF(u_{*}) &= \int_{B_{R}} (\abs{\nabla u_{*}}^{2} + V\abs{u_{*}}^{2})\,\rmd x - \int_{\partial B_{R}} u_{*}\partial_{r}u_{*} \,\rmd S_{x} \\ 
&= \int_{B_{R}} u_{*}(-\Delta u_{*}+Vu_{*})\,\rmd x \\
&= \int_{\{u_{*}=1\}} V\,\rmd x 
\end{aligned}
\end{equation*}
and the proposition follows. 
\end{proof}

\section{Existence and positivity of global minimizers\label{sec:minimizers}}  

First of all, we first establish the following lemma by adapting the ideas in \cite[Lemma~1.1]{GS96FreeBoundaryPotential}. Similar results have previously been used \cite{FP84FB}. 

\begin{lemma}\label{lem:comparison-functionals}
Let $V_{1},V_{2}\in L^{\infty}(\mR^{d})$ be potentials satisfying $V_{1}\ge V_{2}$ in $\mR^{d}$. For any $u_{1},u_{2}\in\mK$, then $\min\{u_{1},u_{2}\} , \max\{u_{1},u_{2}\} \in \mK$ and the following inequality holds:  
\begin{equation}
\mF_{V_{1}}(\min\{u_{1},u_{2}\}) + \mF_{V_{2}}(\max\{u_{1},u_{2}\}) \le \mF_{V_{1}}(u_{1}) + \mF_{V_{2}}(u_{2}). \label{eq:comparison-functionals}
\end{equation}
In particular: 
\begin{enumerate}
\renewcommand{\labelenumi}{\theenumi}
\renewcommand{\theenumi}{\rm (\alph{enumi})} 
\item \label{itm:minimizing-sequence-control} if $u_{2}$ minimizes $\mF_{V_{2}}$ in $\mK$, then $\mF_{V_{1}}(\min\{u_{1},u_{2}\})\le \mF_{V_{1}}(u_{1})$. 
\item \label{itm:lower-boun-control} if $u_{1}$ minimizes $\mF_{V_{1}}$ in $\mK$, then $\mF_{V_{2}}(\max\{u_{1},u_{2}\})\le \mF_{V_{2}}(u_{2})$. 
\item \label{itm:pairwise-comparison} if $V_{1}=V_{2}=V$ and both $u_{1}$ and $u_{2}$ minimize $\mF_{V}$ in $\mK$, then both $\max\{u_{1},u_{2}\}$ and $\min\{u_{1},u_{2}\}$ are also global minimizers of $\mF_{V}$ in $\mK$. 
\end{enumerate}
\end{lemma}

\begin{proof}
Let 
\begin{equation*}
v := \min\{u_{1},u_{2}\}\in\mK \quad \text{and} \quad w := \max\{u_{1},u_{2}\}\in\mK. 
\end{equation*}
First of all, we see that  
\begin{equation*}
\int_{\mR^{d}} \left(\abs{\nabla u_{1}}^{2}+\abs{\nabla u_{2}}^{2}\right)\,\rmd x = \int_{\mR^{d}} \left( \abs{\nabla v}^{2} + \abs{\nabla w}^{2} \right)\,\rmd x. 
\end{equation*}
On the other hand, we see that 
\begin{equation*}
\begin{aligned}
& \int_{\mR^{d}} V_{1}\abs{u_{1}}^{2}\,\rmd x + \int_{\mR^{d}} V_{2}\abs{u_{2}}^{2}\,\rmd x \\ 
& \quad = \int_{\mR^{d}} V_{1}\abs{v}^{2}\,\rmd x + \int_{\mR^{d}} V_{2}\abs{w}^{2}\,\rmd x + \int_{\mR^{d}} V_{1} \overbrace{\left( \abs{u_{1}}^{2} - \abs{v}^{2} \right)}^{\ge\,0} \,\rmd x + \int_{\mR^{d}} V_{2} \left( \abs{u_{2}}^{2} - \abs{w}^{2} \right) \,\rmd x \\ 
& \quad \ge \int_{\mR^{d}} V_{1}\abs{v}^{2}\,\rmd x + \int_{\mR^{d}} V_{2}\abs{w}^{2}\,\rmd x + \int_{\mR^{d}} V_{2} \left( \abs{u_{1}}^{2} - \abs{v}^{2} \right) \,\rmd x + \int_{\mR^{d}} V_{2} \left( \abs{u_{2}}^{2} - \abs{w}^{2} \right) \,\rmd x \\ 
& \quad = \int_{\mR^{d}} V_{1}\abs{v}^{2}\,\rmd x + \int_{\mR^{d}} V_{2}\abs{w}^{2}\,\rmd x. 
\end{aligned}
\end{equation*}
Combining the two equations above yields \eqref{eq:comparison-functionals}. 
For \Cref{lem:comparison-functionals}\ref{itm:minimizing-sequence-control}, since $u_{2}\in\mK$ minimizes $\mF_{V_{2}}$, we have $\mF_{V_{2}}(\max\{u_{1},u_{2}\}) \ge \mF_{V_{2}}(u_{2})$, and the conclusion follows. The proofs of \Cref{lem:comparison-functionals}\ref{itm:lower-boun-control} and \Cref{lem:comparison-functionals}\ref{itm:pairwise-comparison} proceed similarly. 
\end{proof}

For each $\alpha>0$, we define 
\begin{equation*}
g_{\alpha}(r) := r^{-\frac{d-2}{2}}K_{\frac{d-2}{2}}(\alpha r) \quad \text{for all $r>0$.}
\end{equation*}
Using standard properties of modified Bessel functions (see \Cref{sec:Bessel}), one readily verifies that the radially symmetric function 
\begin{equation*}
u(x) := g_{\alpha}(\abs{x}) \quad \text{for all $x\in\mR^{d}$}
\end{equation*}
satisfies the following properties: 
\begin{itemize}
\item $u(x)>0$ for all $x\neq 0$. 
\item $u\in H^{1}(\mR^{d}\setminus\overline{B_{R}})$ for all $R>0$. 
\item $(\Delta-\alpha^{2})u=0$ in $\mR^{d}\setminus\{0\}$. 
\item $u(x)\le C\abs{x}^{-\frac{d-1}{2}}e^{-\alpha\abs{x}}$ for all $\abs{x}\ge 1$. 
\end{itemize}
We now prove the following lemma. 

\begin{lemma}\label{lem:compactness-lemma}
We assume that $V$, $\alpha$ and $R$ satisfy \Cref{assu:1}. 
\begin{enumerate}
\renewcommand{\labelenumi}{\theenumi}
\renewcommand{\theenumi}{\rm (\alph{enumi})} 
\item\label{itm:compactness-lemma-part1} We define 
\begin{equation}
\mK_{0} := \left\{ u\in\mK : u\le u_{0} \right\} ,\quad u_{0}(x):=\frac{g_{\alpha}(\abs{x})}{g_{\alpha}(R)}. \label{eq:u0}
\end{equation} 
Then for each $u\in \mK\setminus\mK_{0}$, there exists $\tilde{u}\in\mK_{0}$ such that $\mF_{V}(\tilde{u}) < \mF_{V}(u)$. 
\item\label{itm:compactness-lemma-part2} For each $a\in\mR$, the space 
\begin{equation*}
\mK_{0}(a) := \left\{ u\in\mK_{0} : \mF_{V}(u) \le a \right\} 
\end{equation*} 
is bounded in $H^{1}(\mR^{d})$ and is compact in $L^{2}(\mR^{d})$. 
\end{enumerate}
\end{lemma}

\begin{remark}\label{rem:weak-conv}
In particular, if $\{u_{j}\}_{j\in\mN}\subset \mK_{0}(a)$, then there exists a subsequence, still denoted by $u_{j}$, such that $u_{j}\rightarrow u$ in $L^{2}(\mR^{d})$-strong and in $H^{1}(\mR^{d})$-weak. Consequently, $\mF_{V}(u)\le\liminf_{j\rightarrow\infty}\mF_{V}(u_{j})$. Using \cite[Theorem~3.12]{Rud87RealComplexAnalysis}, one can choose the subsequence $\{u_{j}\}_{j\in\mN}$ so that $u_{j}(x)\rightarrow u(x)$ for a.e. $x\in\mR^{d}$. 
\end{remark}

\begin{proof}[Proof of \Cref{lem:compactness-lemma}\ref{itm:compactness-lemma-part1}]
By the Lax-Milgram theorem \cite[Corollary~5.8]{Bre11PDE}, there exists a unique minimizer $w\in H^{1}(\overline{B_{R}}^{\complement})$ of the following minimizing problem: 
\begin{equation*}
\begin{aligned}
& \text{minimizing} && \frac{1}{2}\int_{\overline{B_{R}}^{\complement}} (\abs{\nabla v}^{2} + Vv^{2})\,\rmd x \\ 
& \text{subject to} && \left\{ v \in H^{1}(\overline{B_{R}}^{\complement}) : v|_{\partial B_{R}} = u \right\}. 
\end{aligned}
\end{equation*}
We define 
\begin{equation*}
\tilde{u} = \left\{\begin{aligned}
& u && \text{in $B_{R}$,} \\ 
& w && \text{in $\overline{B_{R}}^{\complement}$.} 
\end{aligned}\right. 
\end{equation*}
It remains to show that $\tilde{u}\in\mK_{0}$. 

By the comparison principle (\Cref{lem:comparison-principle}), we therefore conclude that $w\ge 0$ in $\overline{B_{R}}^{\complement}$. 
On the other hand, we define $f:=(V-\alpha^{2})w\ge 0$ in $\overline{B_{R}}^{\complement}$. Since $u_{0}-w\in H^{1}(\overline{B_{R}}^{\complement})$ satisfies 
\begin{equation*}
(\Delta -\alpha^{2})(u_{0}-w) = (-V+\alpha^{2})w \le 0 \text{ in $\overline{B_{R}}^{\complement}$} ,\quad (u_{0}-w)|_{\partial B_{R}} \ge 0, 
\end{equation*}
by comparison principle (\Cref{lem:comparison-principle}), we conclude that $u_{0}-w\ge 0$ in $\overline{B_{R}}^{\complement}$. 
Consequently, we obtain $0 \le w \le u_{0}$ in $\overline{B_{R}}^{\complement}$ and conclude that $w\in \mK_{0}$. Since $u\in\mK\setminus\mK_{0}$, it follows that $u\not\equiv w$ in $\overline{B_{R}}^{\complement}$. This yields the strict inequality 
\begin{equation*}
\frac{1}{2}\int_{\overline{B_{R}}^{\complement}} (\abs{\nabla w}^{2} + Vw^{2})\,\rmd x < \frac{1}{2}\int_{\overline{B_{R}}^{\complement}} (\abs{\nabla u}^{2} + Vu^{2})\,\rmd x
\end{equation*}
and completes the proof of \Cref{lem:compactness-lemma}\ref{itm:compactness-lemma-part1}. 
\end{proof}

\begin{proof}[Proof of \Cref{lem:compactness-lemma}\ref{itm:compactness-lemma-part2}] 
It is not difficult to verify that $\mK_{0}(a)$ is bounded in $H^{1}(\mR^{d})$ and closed in $L^{2}(\mR^{d})$. We now show that $\mK_{0}(a)$ is compact in $L^{2}(\mR^{d})$. 

Let $u\in\mK_{0}(a)$. Given any $\epsilon>0$, there exists $R_{\epsilon}>0$ such that $\norm{u}_{L^{2}(B_{R'}^{\complement})}^{2} < \frac{1}{3}\epsilon$ for all $R' \ge R_{\epsilon}$. 
Since the embedding $H^{1}(B_{R'})\hookrightarrow L^{2}(B_{R'})$ is compact (this is the well-known Rellich-Kondrachov theorem), then we can find $u_{1},\cdots,u_{k(\epsilon,R')}\in \mK_{0}(a)$ such that 
\begin{equation*}
\norm{u-u_{j}}_{L^{2}(B_{R'})}^{2} < \frac{1}{3}\epsilon \quad \text{for some $j\in\{1,\cdots,k(\epsilon,R')\}$.} 
\end{equation*}
Now we see that 
\begin{equation*}
\begin{aligned}
& \norm{u-u_{j}}_{L^{2}(\mR^{d})}^{2} = \norm{u-u_{j}}_{L^{2}(B_{R'})}^{2} + \norm{u-u_{j}}_{L^{2}(B_{R'}^{\complement})}^{2} \\ 
& \quad \le \norm{u-u_{j}}_{L^{2}(B_{R'})}^{2} + \norm{u}_{L^{2}(B_{R'}^{\complement})}^{2} + \norm{u_{j}}_{L^{2}(B_{R'}^{\complement})}^{2} \\ 
& \quad \le \frac{2}{3}\epsilon + \norm{u_{0}}_{L^{2}(B_{R'}^{\complement})}^{2} \quad \text{for all $R' \ge R_{\epsilon}$,} 
\end{aligned}
\end{equation*}
where $u_{0}$ is defined in \eqref{eq:u0}. 
Finally, we observe that the above arguments apply to any choice of $R'\ge R_{\epsilon}$ satisfying 
\begin{equation*}
\norm{u_{0}}_{L^{2}(B_{R'}^{\complement})}^{2} \le \frac{1}{3}\epsilon, 
\end{equation*}
which completes the proof of \Cref{lem:compactness-lemma}\ref{itm:compactness-lemma-part2}. 
\end{proof} 

\begin{remark}
By a slight modification of the arguments in the proof of \Cref{lem:compactness-lemma}\ref{itm:compactness-lemma-part2}, we obtain the following result: For each $s>0$, the inclusion $H^{1}(\mR^{d})\cap L^{2}(\mR^{d},\abs{x}^{2s}\,\rmd x) \subset L^{2}(\mR^{d})$ is compact, where $L^{2}(\mR^{d},\abs{x}^{2s}\,\rmd x)$ be the weighted $L^{2}$-space containing functions $f$ for which $\int_{\mR^{d}}\abs{f(x)}^{2}\abs{x}^{2s}\,\rmd x<\infty$. 
\end{remark}

We are now ready to prove \Cref{thm:well-posedness}. Since the proof is rather lengthy, we divide it into several lemmas. 

\begin{lemma}\label{lem:existence}
We assume that $V$ satisfies \Cref{assu:1}. Then there exists a global minimizer $u_{*}$ of \eqref{eq:main1} in $\mK$. 
\end{lemma}

\begin{proof}[Proof of \Cref{lem:existence}]
We first prove the existence of a positive constant $C$ and a global minimizer $u_{*}$ of \eqref{eq:main1} in $\mK$ satisfying 
\begin{equation}
u_{*}(x) \le C\abs{x}^{-\frac{d-1}{2}} e^{-\alpha\abs{x}} \quad \text{for all $x\in\mR^{d}$.} \label{eq:decay-UB} 
\end{equation}
Let $\{u_{j}\}_{j\in\mN}\subset\mK$ be any sequence such that $\mF(u_{j})\searrow m_{*}$. Using \Cref{lem:compactness-lemma}\ref{itm:compactness-lemma-part1}, one can construct a sequence $\{v_{j}\}_{j\in\mN}\subset\mK_{0}(a)$, with $a = \mF(u_{1})$, such that $\mF(v_{j})\searrow m_{*}$. 
Consequently, using \Cref{rem:weak-conv}, we see that there exists a subsequence, still denoted by $v_{j}$, such that 
\begin{equation*}
v_{j} \rightarrow u_{*} \quad \text{in $L^{2}(\mR^{d})$-strong and in $H^{1}(\mR^{d})$-weak} 
\end{equation*}
and 
\begin{equation*}
m_{*} = \liminf_{j\rightarrow\infty}\mF(v_{j}) \ge \mF(u_{*}), 
\end{equation*}
so that $\mF(u_{*})=m_{*}$, i.e., $u_{*}$ is a global minimizer of \eqref{eq:main1} in $\mK$. 
\end{proof} 

We next derive estimates for every global minimizer of \eqref{eq:main1} in $\mK$. At this stage, it is not yet known how many global minimizers there are.

\begin{lemma}\label{lem:bound} 
We assume that $V$ satisfies \Cref{assu:1}. Then the following holds: 
\begin{enumerate}
\renewcommand{\labelenumi}{\theenumi}
\renewcommand{\theenumi}{\rm (\alph{enumi})} 
\item \label{itm:basic-properties-minimizers} If $u_{*}$ is a global minimizer of \eqref{eq:main1} in $\mK$, then $m_{*}:=\min_{u\in\mK}\mF(u) = \mF(u_{*}) < 0$ and $\{u_{*}=1\}$ has positive measure. 
\item \label{itm:nonempty-contact-set} Recall $\alpha_{0}:= \norm{V_{+}}_{L^{\infty}(\mR^{d})}^{1/2}$. There exist positive constants $c=c(d,V)$, $C=C(R,d,\alpha)$ and $C'=C'(R,d,\alpha,\alpha_{0},\norm{V_{-}}_{L^{\infty}(\mR^{d})})$ such that, for every global minimizer $u_{*}$ of \eqref{eq:main1} in $\mK$, one has \eqref{eq:decay} and \eqref{eq:decay-gradient}. 
\item \label{item:smallest-minimizer} There exist the smallest and largest global minimizers, denoted by $\underline{u}_{*}$ and $\overline{u}_{*}$ respectively, of \eqref{eq:main1} in $\mK$, such that for every global minimizer $u_{*}$ of \eqref{eq:main1} in $\mK$, 
\begin{equation*}
\underline{u}_{*}(x)\le u_{*}(x) \le \overline{u}_{*}(x) \quad \text{for all $x\in\mR^{d}$.}
\end{equation*}  
\item \label{itm:radially-symmetric} If, in addition, $V$ is radially symmetric, then both $\overline{u}_{*}$ and $\underline{u}_{*}$ are radially symmetric. 
\end{enumerate}
\end{lemma}

\begin{proof}[Proof of \Cref{lem:bound}] 
Since the proof is rather lengthy, we divide it into several steps for clarity.

\medskip 

\noindent \emph{Step 1: First, we prove \ref{itm:basic-properties-minimizers}.} 
The first assertion follows from \eqref{eq:nontrivial}. It is straightforward to see that $\{u_{*}=1\}$ has positive measure. Indeed, otherwise, \Cref{rem:simple-obervations} would imply that $m_{*}=0$, contradicting \eqref{eq:nontrivial}. 

\medskip 

\noindent \emph{Step 2: Next, we prove the upper bounds in \eqref{eq:decay} and \eqref{eq:decay-gradient}.} 
The upper bound in \eqref{eq:decay} follows from \Cref{lem:compactness-lemma}\ref{itm:compactness-lemma-part1}. By \Cref{rem:Landis,rem:overpopulated} and the upper bound in \eqref{eq:decay}, it follows that 
\begin{equation*}
\abs{\Delta u_{*}(x)} = \abs{V(x)u_{*}(x)} \le \alpha_{0}^{2}u_{*}(x) \le \alpha_{0}^{2}C(R,d,\alpha)\abs{x}^{-\frac{d-1}{2}}e^{-\alpha\abs{x}} \quad \text{in $B_{R}^{\complement}$.}
\end{equation*}
For any $x$ with $\abs{x}=t\ge \max\{2,R+1\}$, we have 
\begin{equation*}
\abs{\Delta u_{*}(y)} \le \alpha_{0}^{2}C(R,d,\alpha)(t-1)^{-\frac{d-1}{2}}e^{-\alpha(t-1)} \quad \text{for all $y\in B_{1}(x)$.}
\end{equation*}
Using Harnack's inequality (\Cref{lem:Harnack}\ref{itm:Harnack-c}) with $M_{3}=\alpha_{0}^{2}C(R,d,\alpha)(t-1)^{-\frac{d-1}{2}}e^{-\alpha(t-1)}$, we obtain 
\begin{equation*}
\begin{aligned}
\abs{\nabla u_{*}(x)} &\le d \dashint_{\partial B_{1}(x)}u_{*}\,\rmd S + \frac{d\alpha_{0}^{2}C(R,d,\alpha)(t-1)^{-\frac{d-1}{2}}e^{-\alpha(t-1)}}{d+1} \\
& \le C' (t-1)^{-\frac{d-1}{2}}e^{-\alpha(t-1)} \le e^{\alpha}2^{\frac{d-1}{2}}  C't^{-\frac{d-1}{2}}e^{-\alpha t}, 
\end{aligned}
\end{equation*}
since $(t-1)^{-\frac{d-1}{2}}\le 2^{\frac{d-1}{2}}t^{-\frac{d-1}{2}}$ for all $t\ge2$. 

On the other hand, since $\abs{\Delta u_{*}}\le \abs{V(x)u_{*}}\le\norm{V}_{L^{\infty}(\mR^{d})}$, we may apply Harnack's inequality (\Cref{lem:Harnack}\ref{itm:Harnack-c}) with $M_{3}=\norm{V}_{L^{\infty}(\mR^{d})}$ to obtain 
\begin{equation*}
\abs{\nabla u_{*}(x)} \le C(R,d,\alpha,\alpha_{0},\norm{V_{-}}_{L^{\infty}(\mR^{d})}) \quad \text{for all $x\in\mR^{d}$.} 
\end{equation*}
Combining this estimate with the previous ones yields \eqref{eq:decay-gradient}.

\medskip 

\noindent \emph{Step 3: We now prove \ref{item:smallest-minimizer} using \Cref{lem:comparison-functionals}.} 
We now establish the existence of the smallest global minimizers as in \cite[Proposition~1.8]{GS96FreeBoundaryPotential}, similar arguments also work for the largest one. Since $H^{1}(\mR^{d})$ is separable, there is a finite or countably infinite sequence $\{u_{*}^{(j)}\}\subset\mK_{0}(m_{*})$ of the global minimizer of \eqref{eq:main1} that is dense in the set of all minimizer of \eqref{eq:main1}\footnote{Any subset of a separable metric space is also separable \cite[Proposition~3.25]{Bre11PDE}.}.  Define $\tilde{u}_{*}^{(1)}:=u_{*}^{(1)}$ and, inductively for $j\ge 2$, $\tilde{u}_{*}^{(j)}:=\min\{\tilde{u}_{*}^{(j-1)},u_{*}^{(j)}\}$, so that $\tilde{u}_{*}^{(1)}\ge\tilde{u}_{*}^{(2)}\ge\cdots$. From \Cref{lem:comparison-functionals}\ref{itm:pairwise-comparison} we know that $\{\tilde{u}_{*}^{(j)}\}\subset\mK_{0}(m_{*})$ are all global minimizer of \eqref{eq:main1} in $\mK$ satisfying \eqref{eq:decay-UB}. Therefore, using \Cref{rem:weak-conv}, we see that $\underline{u}_{*}:=\inf_{j}\tilde{u}_{*}^{(j)}$ is also a minimizer of \eqref{eq:main1} in $\mK$ and satisfying $u_{*}\ge \underline{u}_{*}$ for every minimizer $u_{*}$ of \eqref{eq:main1} in $\mK$. 

\medskip 

\noindent \emph{Step 4: We now prove the lower bound in \eqref{eq:decay}.} One can appropriately choose $a_{0}=a_{0}(d,V)>0$ such that 
\begin{subequations}\label{eq:lower-barrier} 
\begin{equation}
u_{0}^{\rm rad}(x) := a_{0}(\alpha_{0}\abs{x})^{-\frac{d-2}{2}}K_{\frac{d-2}{2}}(\alpha_{0}\abs{x}) \quad \text{for all $x\in \overline{B_{2R}}^{\complement}$}
\end{equation}
is the unique radially symmetric solution of 
\begin{equation}
(\Delta-\alpha_{0}^{2})u_{0}^{\rm rad}=0 \text{ in $B_{2R}^{\complement}$} ,\quad \left. u_{0}^{\rm rad} \right|_{\partial B_{2R}} = \min_{\partial B_{2R}}\underline{u}_{*} ,\quad u_{0}^{\rm rad} \in H^{1}(\overline{B_{2R}}^{\complement}).  
\end{equation} 
\end{subequations} 
Using \eqref{eq:EL} and the comparison principle (\Cref{lem:comparison-principle}), we see that $\underline{u}_{*}>0$ in $\overline{B_{R}}^{\complement}$. Therefore, one sees that $\left. u_{0}^{\rm rad} \right|_{\partial B_{2R}} = \min_{\partial B_{2R}}\underline{u}_{*} > 0$. Since $\underline{u}_{*}-u_{0}^{\rm rad} \in H^{1}(\overline{B_{2R}}^{\complement})$ satisfies 
\begin{equation*}
(\Delta-\alpha_{0}^{2})(\underline{u}_{*}-u_{0}^{\rm rad})=(V-\alpha_{0}^{2})\underline{u}_{*}\le 0 \text{ in $\overline{B_{2R}}^{\complement}$} , \quad (\underline{u}_{*}-u_{0}^{\rm rad})|_{\partial B_{2R}}\ge 0, 
\end{equation*}
we conclude the lower bound in \eqref{eq:decay} by the comparison principle (\Cref{lem:comparison-principle}). 

\medskip 

\noindent \emph{Step 5: It remains to prove \ref{itm:radially-symmetric}.} 
Let $U$ be any real-valued orthogonal matrix. Since $\mF(u)=\mF(u\circ U)$ for all $u\in\mK$ and $U^{\intercal}=U^{-1}$ is also orthogonal, we have  
\begin{equation*}
\underline{u}_{*}(x)\le \underline{u}_{*}(Ux) \le \underline{u}_{*}(U^{\intercal}Ux)=\underline{u}_{*}(x) \quad \text{for all $x\in\mR^{d}$,} 
\end{equation*}
hence $\underline{u}_{*}(x)= \underline{u}_{*}(Ux)$ for all $x\in\mR^{d}$, which concludes the radially symmetry of $\underline{u}_{*}$. A similar argument applies to $\overline{u}_{*}$ as well. 
\end{proof}

We are now ready to conclude the proof of \Cref{thm:well-posedness}.

\begin{proof}[Proof of \Cref{thm:well-posedness}] 
We now show that $\{\overline{u}_{*}=1\}\subset G$. Otherwise, we can choose $x_{0}\in\{\overline{u}_{*}=1\}\cap G^{\complement}$. Since 
\begin{equation}
\Delta\overline{u}_{*} = V\chi_{\{\overline{u}_{*}<1\}}\overline{u}_{*}\ge0 \quad \text{in $G^{\complement}$,} \label{eq:subharmonic-exterior}
\end{equation}
and $\overline{G}^{\complement}$ is connected, using strong maximum principle (\Cref{lem:comparison-principle}) locally, together with the decay in \eqref{eq:decay}, implies that $x_{0}\in\partial G$. 
Since $\overline{u}_{*}\in C^{1}(\mR^{d})$ attains its maximum at $x_{0}$, we have $\nabla\overline{u}_{*}(x_{0})=0$. On the other hand, from \eqref{eq:subharmonic-exterior}, the interior sphere condition for $G^{\complement}$ and Hopf's lemma (\Cref{cor:Hopf}) imply that $\nabla\overline{u}_{*}(x_{0})\neq0$, yielding a contradiction. 

In view of \Cref{lem:existence,lem:bound}, it remains to prove the uniqueness of the minimizer. Define $w:=\overline{u}_{*}-\underline{u}_{*}\ge0$ in $\mR^{d}$. By \Cref{rem:simple-obervations}, we have 
\begin{equation*}
\int_{\{\underline{u}_{*}=1\}}V\,\rmd x = \int_{\{\overline{u}_{*}=1\}}V\,\rmd x. 
\end{equation*}
Since $\{\underline{u}_{*}=1\}\subset\{\overline{u}_{*}=1\}$, so 
\begin{equation*}
\int_{\{\underline{u}_{*}<1\}\cap\{\overline{u}_{*}=1\}}V\,\rmd x = 0. 
\end{equation*}
Now the inclusion $\{\overline{u}_{*}=1\}\subset G$, implies that $V<0$ a.e. on $\{\underline{u}_{*}<1\}\cap\{\overline{u}_{*}=1\}$, which implies that $\{\underline{u}_{*}<1\}\cap\{\overline{u}_{*}=1\}$ has Lebesgue measure zero. It then follows from \eqref{eq:EL} that
\begin{equation*}
\Delta w = V\chi_{\{\overline{u}_{*}<1\}}w \quad \text{a.e. in $\mR^{d}$.} 
\end{equation*} 
By \Cref{lem:bound}\ref{itm:basic-properties-minimizers}, the set $\{w=0\}$ has positive Lebesgue measure. Therefore, by the unique continuation property for elliptic equations from sets of positive measure (see \cite{Regbaoiu2001MUCP} for $d\ge3$ and \cite{GL93MUCP} for $d=2$), we conclude that $w=0$ in $\mR^{d}$. This completes the uniqueness proof. 
Finally, if, in addition, $V$ is radially symmetric, then \Cref{lem:bound}\ref{itm:radially-symmetric} implies that the unique minimizer is also radially symmetric. 
\end{proof}


Having completed the proof of the existence and uniqueness of a minimizer, we now prove \Cref{thm:positivity} concerning the positivity of global minimizers, using Hopf's lemma (\Cref{cor:Hopf}).

\begin{proof}[Proof of \Cref{thm:positivity}] 
Using \Cref{thm:Lewy-Stampacchia}, we then have
\begin{equation*}
\Delta u_{*} \le Vu_{*} +V_-u_*=V_+u_*= 0 \quad \text{in $G$.} 
\end{equation*} 
Since $G$ is connected, the minimum principle for superharmonic functions, together with the fact that $u_{*}$ does not attain a local minimum in $G$, implies that $u_{*}>0$ for all $x\in G$: Otherwise, if $u_{*}(x_{0})=0$ for some $x_{0}\in G$, then the comparison principle (\Cref{lem:comparison-principle}) implies that $u_{*}=0$ in $G$, contradicting to \Cref{rem:Landis}. 

We now prove that $u_{*}>0$ on $\partial G$. Suppose, for the sake of contradiction, that there exists $x_{0}\in\partial G$ such that $u(x_{0})=0$. By the interior sphere condition at $x_{0}\in\partial G$, there exist $a\in G$ and $r>0$ such that $B_{r}(a)\subset G$ and $\partial B_{r}(a)\cap\partial G=\{x_{0}\}$. Let $\nu$ denote the unit vector that points from $a$ to $x_{0}$. Applying Hopf's lemma (\Cref{cor:Hopf}), we obtain $\nu\cdot\nabla u_{*}(x_{0})<0$. Since $u_{*}\in C^{1}(\mR^{d})$, for a sufficiently small $\epsilon>0$ we then have $u_{*}(x_{0}+\epsilon\nu)<0$, which contradicts $u_{*}\ge 0$ in $\mR^{d}$.

Finally, since $\overline{G}^{\complement}$ is connected, applying the comparison principle (\Cref{lem:comparison-principle}) for the operator $\Delta-V$ in $\overline{G}^{\complement}$, and noting that $u_{*}$ does not attain a local minimum there, we conclude that $u_{*}>0$ in $\overline{G}^{\complement}$. 
\end{proof}

\section{Radially symmetric potential case\label{sec:analytic-examples}}

We first present a proof of \Cref{thm:radially-symmetric}. 

\begin{proof}[Proof of \Cref{thm:radially-symmetric}] 
Define\footnote{The upper bound for $R_{*}$ follows from \Cref{thm:well-posedness}.} 
\begin{equation*}
R_{*}:=\max\{r:u_{*}(re_{1})=1\}\in(0,R_{0}) ,\quad r_{*}:=\min\{r:u_{*}(re_{1})=1\}. 
\end{equation*}
First, we show that $r_{*}=0$: Otherwise, if $r_{*}>0$, then 
\begin{equation*}
0 > \int_{B_{r_{*}}}Vu_{*}\,\rmd x = \int_{B_{r_{*}}}\Delta u_{*}\,\rmd x = \int_{\partial B_{r_{*}}}\partial_{r}u_{*}\,\rmd S = 0, 
\end{equation*}
which is a contradiction. 

We show that $\{u_{*}=1\}=\overline{B_{R_{*}}}\subset B_{R_{0}}$: Otherwise, there exist $r_{0}$ and $r_{1}$ such that $0\le r_{0}<r_{1}\le R_{*}$ such that 
\begin{equation*}
u_{*}|_{\partial B_{r_{0}}} = u_{*}|_{\partial B_{r_{1}}} = 1 ,\quad u_{*}|_{B_{r_{1}}\setminus\overline{B_{r_{0}}}}<1. 
\end{equation*} 
Since $r_{0}\le R_{*} <R_{0}$, we compute  
\begin{equation*}
0 > \int_{B_{r_{1}}\setminus\overline{B_{r_{0}}}}Vu_{*}\,\rmd x = \int_{\partial B_{r_{1}}} \partial_{r}u_{*}\,\rmd S - \int_{\partial B_{r_{0}}} \partial_{r}u_{*}\,\rmd S = 0, 
\end{equation*}
which is a contradiction. 

It is easy to see that $\partial_{r}u_{*}< 0$ in $\overline{B_{R_{0}}}\setminus\overline{B_{R_{*}}}$ as well: For any $\tilde{R}\in(R_{*},R_{0}]$, since $u_{*}>0$ in $\mR^{d}$, we have 
\begin{equation*}
\begin{aligned}
0 &> \int_{B_{\tilde{R}}\setminus\overline{B_{R_{*}}}}Vu_{*}\,\rmd x = \int_{B_{\tilde{R}}\setminus\overline{B_{R_{*}}}} \Delta u_{*}\,\rmd x \\
&= \int_{\partial(B_{\tilde{R}}\setminus\overline{B_{R_{*}}})}\partial_{\nu}u_{*}\,\rmd S = \int_{\partial B_{\tilde{R}}}\partial_{r}u_{*}\,\rmd S = d\omega_{d}\tilde{R}^{d-1} \partial_{r}u_{*}|_{\partial B_{\tilde{R}}}. 
\end{aligned}
\end{equation*}
Finally, since 
\begin{equation*}
\frac{1}{r^{d-1}}\partial_{r}\left(r^{d-1}\partial_{r}u_{*}\right) = \Delta u_{*} = Vu_{*} \ge 0 \quad \text{in $\overline{B_{R_{0}}^{\complement}}$,} 
\end{equation*}
then $r^{d-1}\partial_{r}u_{*}$ is nondecreasing. If there exists $\tilde{R}>R_{0}$ such that $\partial_{r}u_{*}\ge0$ on $\partial B_{\tilde{R}}$, then $\partial_{r}u_{*}\ge0$ in $B_{\tilde{R}}^{\complement}$. In view of \eqref{eq:decay}, this would imply that $u_{*}=0$ in $B_{\tilde{R}}^{\complement}$, contradicting the lower bound in \eqref{eq:decay}. This completes the proof. 
\end{proof}

We now present an example for the case in which $V$ is locally constant.

\begin{proposition}\label{lem:radially-symmetric-minimizer} 
Let $R>0$. We consider the potential $V_{*}\in L^{\infty}(\mR^{d})$ given by  
\begin{equation*}
V_{*}(x) = \alpha^{2} \text{ for all $x\in B_{R}^{\complement}$} ,\quad V_{*}(x) = -\beta^{2} \text{ for all $x\in B_{R}$} 
\end{equation*}
for any $\beta^{2} > \lambda^{*}(B_{R})$. Then there exists a unique global minimizer $u_{*}^{\rm rad}$ of $\mF_{*} \equiv \mF_{V_{*}}$ given by 
\begin{equation}
u_{*}^{\rm rad}(x) = \left\{\begin{aligned}
& 1 && \text{for all $x\in B_{R_{*}}$,} \\ 
& a (\beta\abs{x})^{-\frac{d-2}{2}}J_{\frac{d-2}{2}}(\beta\abs{x}) + b (\beta\abs{x})^{-\frac{d-2}{2}}Y_{\frac{d-2}{2}}(\beta\abs{x}) && \text{for all $x\in B_{R}\setminus B_{R_{*}}$,} \\ 
& c (\alpha\abs{x})^{-\frac{d-2}{2}}K_{\frac{d-2}{2}}(\alpha\abs{x}) && \text{for all $x\in B_{R}^{\complement}$,} 
\end{aligned}\right. \label{eq:global-minimizer-result} 
\end{equation}
with 
\begin{equation*}
\begin{aligned} 
a &= -\frac{\pi}{2}(\beta R_{*})^{\frac{d}{2}}Y_{\frac{d}{2}}(\beta R_{*}), \quad b = \frac{\pi}{2}(\beta R_{*})^{\frac{d}{2}}J_{\frac{d}{2}}(\beta R_{*}), \\ 
c &= \frac{1}{(\alpha R)^{-\frac{d-2}{2}}K_{\frac{d-2}{2}}(\alpha R)} \left( \begin{aligned} 
&-\frac{\pi}{2}(\beta R_{*})^{\frac{d}{2}}Y_{\frac{d}{2}}(\beta R_{*})(\beta R)^{-\frac{d-2}{2}}J_{\frac{d-2}{2}}(\beta R) \\
& \quad+ \frac{\pi}{2}(\beta R_{*})^{\frac{d}{2}}J_{\frac{d}{2}}(\beta R_{*}) (\beta R)^{-\frac{d-2}{2}}Y_{\frac{d-2}{2}}(\beta R) 
\end{aligned}\right), 
\end{aligned} 
\end{equation*}
where $R_{*}$ is the largest zero in $(0,R)$ of 
\begin{equation}
-Y_{\frac{d}{2}}(\beta R_{*}) \left( \frac{J_{\frac{d-2}{2}}(\beta R)}{K_{\frac{d-2}{2}}(\alpha R)} - \frac{\beta J_{\frac{d}{2}}(\beta R)}{\alpha K_{\frac{d}{2}}(\alpha R)} \right) = J_{\frac{d}{2}}(\beta R_{*}) \left( \frac{\beta Y_{\frac{d}{2}}(\beta R)}{\alpha K_{\frac{d}{2}}(\alpha R)} - \frac{Y_{\frac{d-2}{2}}(\beta R)}{K_{\frac{d-2}{2}}(\alpha R)} \right). \label{eq:solution-R-star}
\end{equation}  
In addition, it satisfies $\mF_{*}(u_{*}^{\rm rad})=-\beta^{2}\omega_{d}R_{*}^{d}$. 
\end{proposition}

Prior to the proof of \Cref{lem:radially-symmetric-minimizer}, we establish the following lemma. 

\begin{lemma}\label{lem:general-solution} 
Let 
\begin{equation*}
v(r):=\left\{\begin{aligned}
& a_{0}(\beta r)^{-\frac{d-2}{2}}J_{\frac{d-2}{2}}(\beta r) + b_{0}(\beta r)^{-\frac{d-2}{2}}Y_{\frac{d-2}{2}}(\beta r) && \text{for $0<r<R$,} \\ 
& (\alpha r)^{-\frac{d-2}{2}}K_{\frac{d-2}{2}}(\alpha r) && \text{for $r>R$.} 
\end{aligned}\right. 
\end{equation*}
Then $v\in C^{1}((0,\infty))$ if and only if 
\begin{equation*}
\left(\begin{array}{c}a_{0} \\ b_{0} \end{array}\right) = -(\beta R)^{\frac{d-2}{2}} \frac{\pi\beta R}{2} \left( \begin{array}{cc} Y_{\frac{d}{2}}(\beta R) & -Y_{\frac{d-2}{2}}(\beta R) \\ -J_{\frac{d}{2}}(\beta R) & J_{\frac{d-2}{2}}(\beta R) \end{array} \right)\left(\begin{array}{c}(\alpha R)^{-\frac{d-2}{2}}K_{\frac{d-2}{2}}(\alpha R) \\ \alpha (\alpha R)^{-\frac{d-2}{2}}K_{\frac{d}{2}}(\alpha R) \end{array}\right). 
\end{equation*}
\end{lemma}

\begin{proof}[Proof of \Cref{lem:general-solution}]
Note that $v\in C^{1}((0,\infty))$ if and only if $(a_{0},b_{0})$ solves 
\begin{equation*}
\left\{\begin{aligned}
& a_{0} (\beta R)^{-\frac{d-2}{2}}J_{\frac{d-2}{2}}(\beta R) + b_{0} (\beta R)^{-\frac{d-2}{2}}Y_{\frac{d-2}{2}}(\beta R) = (\alpha R)^{-\frac{d-2}{2}}K_{\frac{d-2}{2}}(\alpha R), \\ 
& a_{0} \beta (\beta R)^{-\frac{d-2}{2}}J_{\frac{d}{2}}(\beta R) + b_{0} \beta (\beta R)^{-\frac{d-2}{2}}Y_{\frac{d}{2}}(\beta R) =  \alpha (\alpha R)^{-\frac{d-2}{2}}K_{\frac{d}{2}}(\alpha R), 
\end{aligned}\right.
\end{equation*}
equivalently, 
\begin{equation*}
\left(\begin{array}{c}a_{0} \\ b_{0} \end{array}\right) = (\beta R)^{\frac{d-2}{2}} \left( \begin{array}{cc} J_{\frac{d-2}{2}}(\beta R) & Y_{\frac{d-2}{2}}(\beta R) \\ J_{\frac{d}{2}}(\beta R) & Y_{\frac{d}{2}}(\beta R) \end{array} \right)^{-1} \left(\begin{array}{c}(\alpha R)^{-\frac{d-2}{2}}K_{\frac{d-2}{2}}(\alpha R) \\ \alpha (\alpha R)^{-\frac{d-2}{2}}K_{\frac{d}{2}}(\alpha R) \end{array}\right). 
\end{equation*}
From \eqref{eq:Wronskian-Bessel} we obtain 
\begin{equation*}
\det \left( \begin{array}{cc} J_{\frac{d-2}{2}}(\beta R) & Y_{\frac{d-2}{2}}(\beta R) \\ J_{\frac{d}{2}}(\beta R) & Y_{\frac{d}{2}}(\beta R) \end{array} \right) = -\frac{2}{\pi\beta R}. 
\end{equation*}
This completes the proof of the lemma. 
\end{proof}

We now proceed to the proof of \Cref{lem:radially-symmetric-minimizer}. 

\begin{proof}[Proof of \Cref{lem:radially-symmetric-minimizer}] 
By \Cref{thm:positivity,thm:radially-symmetric}, there exists a unique positive global minimizer $u_{*}^{\rm rad}\in C^{1}(\mR^{d})$ of $\mF_{*}=\mF_{V_{*}}$, which is radially symmetric. By \Cref{thm:radially-symmetric}, it follows that 
\begin{equation}
\left\{\begin{aligned}
& u_{*}^{\rm rad} = 1 && \text{in $B_{R_{*}}$,} \\ 
& \Delta u_{*}^{\rm rad} + \beta^{2}u_{*}^{\rm rad} = 0 && \text{in $B_{R}\setminus \overline{B_{R_{*}}}$,} \\ 
& \Delta u_{*}^{\rm rad} - \alpha^{2}u_{*}^{\rm rad} = 0 && \text{in $\overline{B_{R_{*}}}^{\complement}$,} \\ 
& \lim_{\lvert x \rvert \to \infty} u_{*}^{\rm rad}(x) = 0 \\ 
& u_{*}^{\rm rad} > 0 && \text{in $\mR^{d}$}
\end{aligned}\right. 
\end{equation}
for some $0< R_{*} < R$ satisfying $v'(R_{*})=0$, where $v$ is defined in \Cref{lem:general-solution}. 
It is easy to see that $R_{*}$ is the maximum value which solves $v'(R_{*})=0$ in $(0,R)$, equivalently, $R_{*}$ is the largest zero in $(0,R)$ of \eqref{eq:solution-R-star}: If not, there exists $\tilde{R}_{*} \in (R_{*},R)$ with $v'(\tilde{R}_{*})=0$, then 
\begin{equation*}
0 < \int_{B_{\tilde{R}_{*}}\setminus\overline{B_{R_{*}}}} u_{*}^{\rm rad}\,\rmd x = -\frac{1}{\beta^{2}}\int_{B_{\tilde{R}_{*}}\setminus\overline{B_{R_{*}}}}\Delta u_{*}^{\rm rad}\,\rmd x = -\frac{1}{\beta^{2}}\int_{\partial(B_{\tilde{R}_{*}}\setminus\overline{B_{R_{*}}})} \partial_{\nu}u_{*}^{\rm rad}\,\rmd S = 0, 
\end{equation*}
which leads a contradiction. 

We now check that $u_{*}^{\rm rad} > 0$ in $\mR^{d}$. It suffices to check that the derivative of the mapping  
\begin{equation}
t \mapsto -Y_{\frac{d}{2}}(\beta R_{*}) (\beta t)^{\frac{d-2}{2}}J_{\frac{d-2}{2}}(\beta t) + J_{\frac{d}{2}}(\beta R_{*}) (\beta t)^{\frac{d-2}{2}}Y_{\frac{d-2}{2}}(\beta t) \label{eq:to-do} 
\end{equation} 
is negative for all $t\in(R_{*},R)$. Otherwise, then one can find $t\in(R_{*},R)$ such that 
\begin{equation}
-Y_{\frac{d}{2}}(\beta R_{*}) J_{\frac{d}{2}}(\beta t) + J_{\frac{d}{2}}(\beta R_{*}) Y_{\frac{d}{2}}(\beta t) = 0. \label{eq:contrary-assumption}
\end{equation}
Since $R_{*}$ is the greatest solution of \eqref{eq:solution-R-star} within $(0, R)$, then each $t \in (R_{*},R)$ is not a solution to \eqref{eq:solution-R-star}, following 
\begin{equation}
- Y_{\frac{d}{2}}(\beta t) \left( \frac{J_{\frac{d-2}{2}}(\beta R)}{K_{\frac{d-2}{2}}(\alpha R)} - \frac{\beta J_{\frac{d}{2}}(\beta R)}{\alpha K_{\frac{d}{2}}(\alpha R)} \right) \neq  J_{\frac{d}{2}}(\beta t) \left( \frac{\beta Y_{\frac{d}{2}}(\beta R)}{\alpha K_{\frac{d}{2}}(\alpha R)} - \frac{Y_{\frac{d-2}{2}}(\beta R)}{K_{\frac{d-2}{2}}(\alpha R)} \right). \label{eq:observation}
\end{equation}

\medskip 

\noindent \emph{Case 1.} If $Y_{\frac{d}{2}}(\beta R_{*})=0$, then $J_{\frac{d}{2}}(\beta R_{*})\neq 0$ and $Y_{\frac{d}{2}}(\beta t)=0$. From \eqref{eq:observation} we see that 
\begin{equation*}
J_{\frac{d}{2}}(\beta t) \left( \frac{\beta Y_{\frac{d}{2}}(\beta R)}{\alpha K_{\frac{d}{2}}(\alpha R)} - \frac{Y_{\frac{d-2}{2}}(\beta R)}{K_{\frac{d-2}{2}}(\alpha R)} \right) \neq 0. 
\end{equation*}
Therefore we see that $J_{\frac{d}{2}}(\beta t)\neq 0$ and 
\begin{equation*}
\frac{\beta Y_{\frac{d}{2}}(\beta R)}{\alpha K_{\frac{d}{2}}(\alpha R)} - \frac{Y_{\frac{d-2}{2}}(\beta R)}{K_{\frac{d-2}{2}}(\alpha R)} \neq 0. 
\end{equation*}
Now 
\begin{equation*}
J_{\frac{d}{2}}(\beta R_{*})\left( \frac{\beta Y_{\frac{d}{2}}(\beta R)}{\alpha K_{\frac{d}{2}}(\alpha R)} - \frac{Y_{\frac{d-2}{2}}(\beta R)}{K_{\frac{d-2}{2}}(\alpha R)} \right) \neq 0 = - Y_{\frac{d}{2}}(\beta R_{*}) \left( \frac{J_{\frac{d-2}{2}}(\beta R)}{K_{\frac{d-2}{2}}(\alpha R)} - \frac{\beta J_{\frac{d}{2}}(\beta R)}{\alpha K_{\frac{d}{2}}(\alpha R)} \right),  
\end{equation*}
which contradicts with \eqref{eq:solution-R-star}.  

\medskip 

\noindent \emph{Case 2.} If $Y_{\frac{d}{2}}(\beta R_{*})\neq 0$, then\footnote{Otherwise, if $Y_{\frac{d}{2}}(\beta t)= 0$, from \eqref{eq:contrary-assumption} we will obtain $J_{\frac{d}{2}}(\beta t) = 0$, this is a contradiction.} $Y_{\frac{d}{2}}(\beta t)\neq 0$. Now we see that 
\begin{equation*}
\begin{aligned} 
& - Y_{\frac{d}{2}}(\beta R_{*}) Y_{\frac{d}{2}}(\beta t) \left( \frac{J_{\frac{d-2}{2}}(\beta R)}{K_{\frac{d-2}{2}}(\alpha R)} - \frac{\beta J_{\frac{d}{2}}(\beta R)}{\alpha K_{\frac{d}{2}}(\alpha R)} \right) \\
& \quad \neq Y_{\frac{d}{2}}(\beta R_{*}) J_{\frac{d}{2}}(\beta t) \left( \frac{\beta Y_{\frac{d}{2}}(\beta R)}{\alpha K_{\frac{d}{2}}(\alpha R)} - \frac{Y_{\frac{d-2}{2}}(\beta R)}{K_{\frac{d-2}{2}}(\alpha R)} \right) \\
& \quad \equiv Y_{\frac{d}{2}}(\beta t) J_{\frac{d}{2}}(\beta R_{*}) \left( \frac{\beta Y_{\frac{d}{2}}(\beta R)}{\alpha K_{\frac{d}{2}}(\alpha R)} - \frac{Y_{\frac{d-2}{2}}(\beta R)}{K_{\frac{d-2}{2}}(\alpha R)} \right), 
\end{aligned}
\end{equation*}
which gives 
\begin{equation*}
- Y_{\frac{d}{2}}(\beta R_{*}) \left( \frac{J_{\frac{d-2}{2}}(\beta R)}{K_{\frac{d-2}{2}}(\alpha R)} - \frac{\beta J_{\frac{d}{2}}(\beta R)}{\alpha K_{\frac{d}{2}}(\alpha R)} \right) \neq J_{\frac{d}{2}}(\beta R_{*}) \left( \frac{\beta Y_{\frac{d}{2}}(\beta R)}{\alpha K_{\frac{d}{2}}(\alpha R)} - \frac{Y_{\frac{d-2}{2}}(\beta R)}{K_{\frac{d-2}{2}}(\alpha R)} \right), 
\end{equation*}
which contradicts with \eqref{eq:solution-R-star}. 

\medskip 

Combining these two cases, we conclude \eqref{eq:to-do}, and we know that $u_{*}^{\rm rad} > 0$. Finally, Using \Cref{rem:simple-obervations}, one sees that 
\begin{equation*}
\mF_{*}(u_{*}^{\rm rad}) = -\beta^{2} \int_{B_{R_{*}}} \, \rmd x = -\beta^{2}\omega_{d}R_{*}^{d}, 
\end{equation*}
and we complete the proof. 
\end{proof}

\begin{example}[see \Cref{fig:radially-symmetric-minimizer} below]\label{exa:minimizer}
When $d=3$, the radially symmetric global minimizer $u_{*}^{\rm rad}$ from \Cref{lem:radially-symmetric-minimizer} can be expressed using elementary functions: 
\begin{equation}
u_{*}^{\rm rad}(x) = \left\{\begin{aligned}
& 1 && \text{for all $x\in B_{R_{*}}$,} \\ 
& a \sqrt{\frac{2}{\pi}} \frac{\sin(\beta\abs{x})}{\beta\abs{x}} - b \sqrt{\frac{2}{\pi}} \frac{\cos(\beta\abs{x})}{\beta\abs{x}} && \text{for all $x\in B_{R}\setminus\overline{B_{R_{*}}}$,} \\ 
& c \sqrt{\frac{\pi}{2}} \frac{e^{-\alpha\abs{x}}}{\alpha\abs{x}} && \text{for all $x\in B_{R}^{\complement}$.} 
\end{aligned}\right. \label{eq:global-minimizer-result-3D} 
\end{equation}
with 
\begin{equation*}
\begin{aligned} 
a \sqrt{\frac{2}{\pi}} &= (\beta R_{*})\sin(\beta R_{*}) + \cos(\beta R_{*}) ,\quad b \sqrt{\frac{2}{\pi}} = \sin(\beta R_{*}) -(\beta R_{*}) \cos(\beta R_{*}) \\ 
c \sqrt{\frac{\pi}{2}} &= \frac{\alpha R}{e^{-\alpha R}} \left( a \sqrt{\frac{2}{\pi}} \frac{\sin(\beta R)}{\beta R} - b \sqrt{\frac{2}{\pi}} \frac{\cos(\beta R)}{\beta R} \right), 
\end{aligned} 
\end{equation*}
where $R_{*}$ is the greatest solution of 
\begin{equation*}
\begin{aligned} 
& \left( \sin(\beta R_{*}) + \frac{\cos(\beta R_{*})}{\beta R_{*}} \right) \left( \sin(\beta R) - \frac{\beta (\frac{\sin(\beta R)}{\beta R}-\cos(\beta R))}{\alpha (\frac{1}{\alpha R}+1)} \right) \\
& \quad = \left( \frac{\sin(\beta R_{*})}{\beta R_{*}} - \cos(\beta R_{*}) \right) \left( \cos(\beta R) - \frac{\beta(\sin(\beta R) + \frac{\cos(\beta R)}{\beta R})}{\alpha (\frac{1}{\alpha R}+1)} \right). 
\end{aligned} 
\end{equation*} 
within $(0, R)$. 
\end{example}

\begin{figure}[H]
\begin{subfigure}[b]{0.35\textwidth}
\centering
\includegraphics[width=\linewidth]{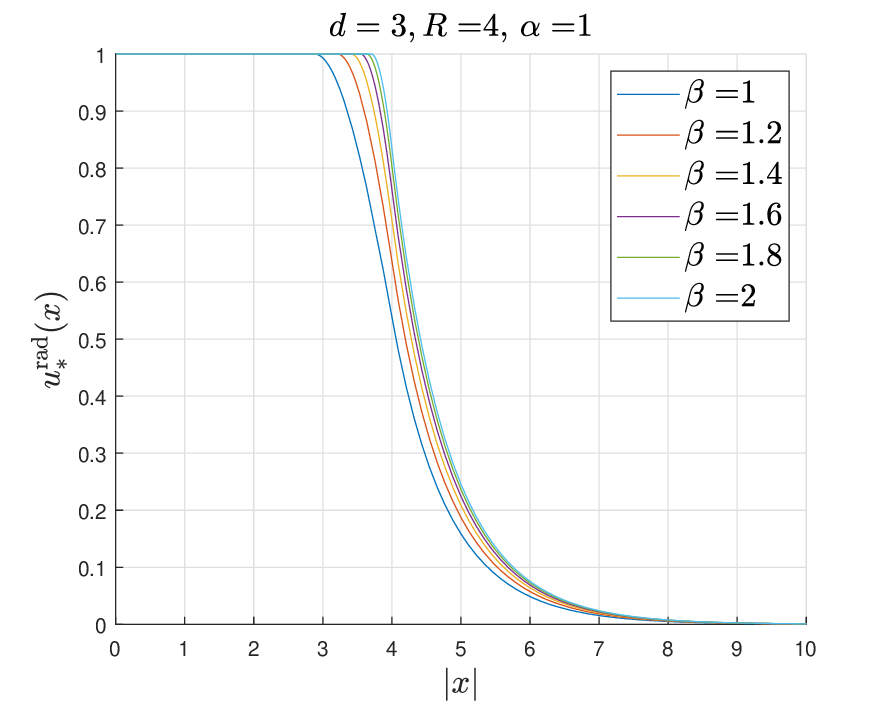}
\end{subfigure}
\begin{subfigure}[b]{0.35\textwidth}
\centering
\includegraphics[width=\linewidth]{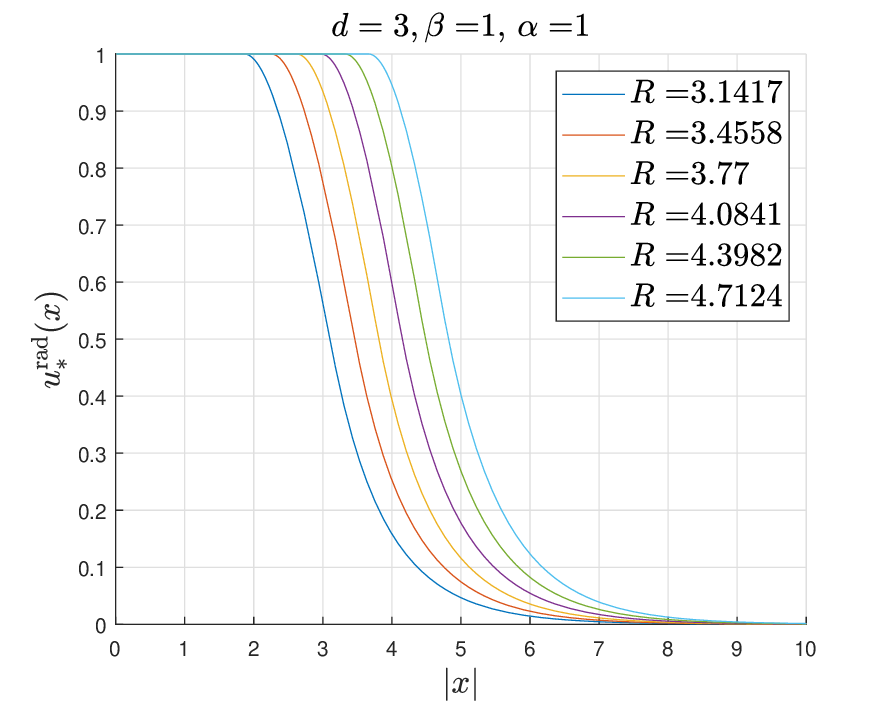}
\end{subfigure}
\begin{subfigure}[b]{0.35\textwidth}
\centering
\includegraphics[width=\linewidth]{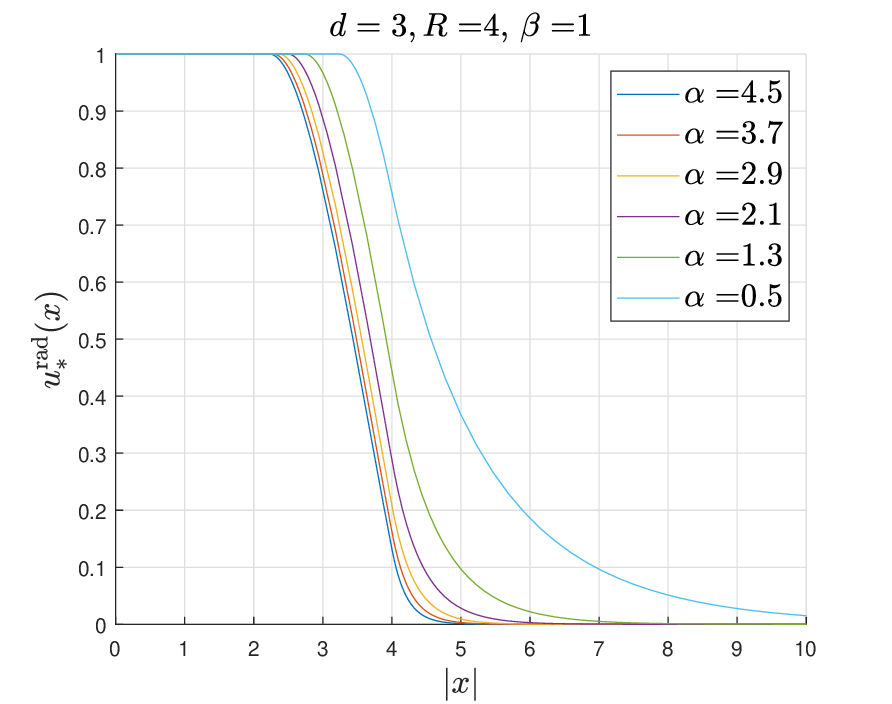}
\end{subfigure}

\caption{Plot of \Cref{exa:minimizer}}
\label{fig:radially-symmetric-minimizer}
\end{figure}

\section{Numerical examples: A heuristic gradient descent method\label{sec:numerical}}

This section presents numerical examples of minimizers of \eqref{eq:main1} described in \Cref{thm:well-posedness}, as analytical examples are difficult to obtain. Here, we additionally assume that there exists $\tilde{r}>0$ such that $B_{\tilde{r}}\subset G$ and $\beta^{2} > \lambda^{*}(B_{\tilde{r}}) = j_{\frac{d-2}{2},1}^{2}\tilde{r}^{-2}$, where $j_{\alpha,m}$ is the $m^{\rm th}$ positive zero of the Bessel function $J_{\alpha}$ of the first kind. We  consider the potential $\tilde{V}\in L^{\infty}(\mR^{d})$ given by 
\begin{equation*}
\tilde{V}=\alpha_{0}^{2}\equiv \sup_{\mR^{d}}V \text{ for all $x\in B_{\tilde{r}}^{\complement}$} ,\quad \tilde{V}=-\beta^{2} \text{ for all $x\in B_{\tilde{r}}$.} 
\end{equation*}
It is easy to see that $V\le\tilde{V}$ in $\mR^{d}$. Arguing as in \Cref{lem:radially-symmetric-minimizer} we see that $\mF_{\tilde{V}}$ admits a radially symmetric global minimizer $\tilde{u}^{\rm rad}$ in $\mK$ satisfying  
\begin{equation*}
\tilde{u}^{\rm rad}(x) \sim c\abs{x}^{-\frac{d-1}{2}}e^{-\alpha_{0}\abs{x}} \quad \text{as $\abs{x}\rightarrow\infty$.}
\end{equation*}
for some positive constant $c>0$. 

For simplicity, we restrict ourselves for the case when $d=2$, and assume that $V$ takes the form 
\begin{equation*}
V(x) = \alpha^{2} \quad \text{for a.e. $x\in G^{\complement}$} ,\quad V(x) = -\beta^{2} \quad \text{for a.e. $x\in G$}, 
\end{equation*}
provided $G \supset B_{\tilde{r}}$ and $\beta^{2} > \lambda^{*}(B_{\tilde{r}}) = j_{0,1}^{2}\tilde{r}^{-2}$ for some $\tilde{r}>0$. Write $R>0$ such that $G\subset B_{R}$. 
From \Cref{lem:comparison-functionals}, we now see that the unique minimizer $u_{*} \in C^{1}$ of \eqref{eq:main1} satisfies  
\begin{equation}
u_{*}^{\rm lower} \le u_{*} \le u_{*}^{\rm upper}, \label{eq:minimziers-to-estimate}
\end{equation}
where $u_{*}^{\rm lower}$ and $u_{*}^{\rm upper}$ radially symmetric function described in \Cref{lem:radially-symmetric-minimizer}: 
\begin{equation*}
u_{*}^{\rm lower}(x) = \left\{\begin{aligned}
& 1 \quad \text{for all $x\in B_{\tilde{r}_{*}}$}, \\ 
& -\frac{\pi}{2}(\beta \tilde{r}_{*})Y_{1}(\beta \tilde{r}_{*})J_{0}(\beta\abs{x}) + \frac{\pi}{2}(\beta \tilde{r}_{*}) J_{1}(\beta \tilde{r}_{*})Y_{0}(\beta\abs{x}) \quad \text{for all $x\in B_{\tilde{r}}\setminus\overline{B_{\tilde{r}_{*}}}$} \\ 
& \frac{1}{K_{0}(\alpha \tilde{r})}\left( -\frac{\pi}{2}(\beta \tilde{r}_{*})Y_{1}(\beta \tilde{r}_{*})J_{0}(\beta \tilde{r}) + \frac{\pi}{2}(\beta \tilde{r}_{*}) J_{1}(\beta \tilde{r}_{*})Y_{0}(\beta \tilde{r}) \right) K_{0}(\alpha \abs{x}) \quad \text{o.w.} 
\end{aligned}\right.
\end{equation*}
where $\tilde{r}_{*}$ is the largest zero in $(0,\tilde{r})$ of 
\begin{equation*}
-Y_{1}(\beta \tilde{r}_{*}) \left( \frac{J_{0}(\beta \tilde{r})}{K_{0}(\alpha \tilde{r})} - \frac{\beta J_{1}(\beta \tilde{r})}{\alpha K_{1}(\alpha \tilde{r})} \right) = J_{1}(\beta \tilde{r}_{*}) \left( \frac{\beta Y_{1}(\beta \tilde{r})}{\alpha K_{1}(\alpha \tilde{r})} - \frac{Y_{0}(\beta \tilde{r})}{K_{0}(\alpha \tilde{r})} \right). 
\end{equation*}
and $u_{*}^{\rm upper}$ has a similar form, with $(\tilde{r}_{*}, \tilde{r})$ replaced by $(R_{*}, R)$. We implement our numerical experiment with $\alpha=5$, $\beta=3$, $\tilde{r}=1$ and $R=3$, see \Cref{fig:LB-UB}. 

\begin{figure}[H]
\centering
\includegraphics[width=0.75\linewidth]{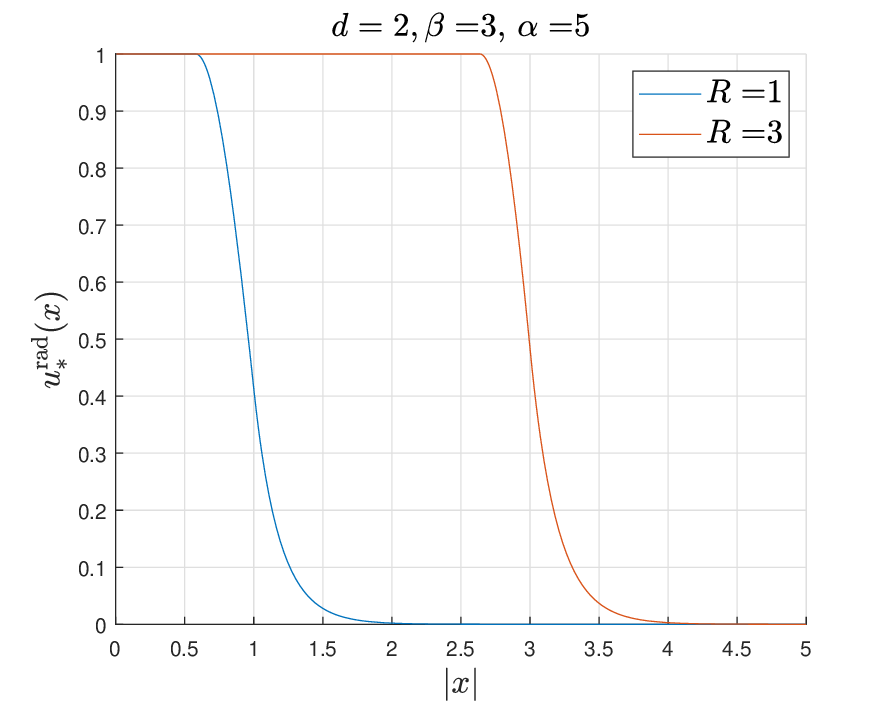}
\caption{Plot of $u_{*}^{\mathrm{lower}}$ and $u_{*}^{\mathrm{upper}}$ as functions of $\lvert x \rvert$}
\label{fig:LB-UB}
\end{figure}

By \eqref{eq:EL} and \eqref{eq:minimziers-to-estimate}, we see that 
\begin{equation*}
(-\Delta + V\chi_{\{u_{*}<1\}})u_{*} = 0 \text{ in $B_{T}$} ,\quad u_{*}|_{\partial B_{T}} \approx 0  
\end{equation*}
for large $T>0$. We consider a heuristic gradient descent method based on the implicit Euler scheme: Given initial guess $u^{(1)}=u_{*}^{\rm upper}$, we construct a sequence of functions $\{u^{(m)}\}_{m\in\mN}$, given by 
\begin{equation*}
u^{(m)} := \max\{\min\{\tilde{u}^{(m)},u_{*}^{\rm upper}\},u_{*}^{\rm lower}\}, 
\end{equation*}
where 
\begin{equation}
\tilde{u}^{(m+1)} = u^{(m)} - \tau \partial_{u}\mF(u)|_{u=\tilde{u}^{(m+1)}}, \quad \tilde{u}^{(m+1)}|_{\partial B_{T}} = 0, \label{eq:PDE-iteration}
\end{equation}
and 
\begin{equation*}
\partial_{u}\mF(u) = (-\Delta + V)u. 
\end{equation*}
It is possible to solve \eqref{eq:PDE-iteration} using \href{https://www.mathworks.com/help/pde/index.html}{\textsc{Matlab PDE toolbox}}. We choose $T=5$, maximum FEM size $\le 0.1$ and perform $50$ iterations. We present our numerical simulation in \Cref{fig:minimizer}. However, the Euler-Lagrange equation \eqref{eq:EL} is nonlinear and admits a trivial solution, making the analysis of the numerical scheme challenging. We leave this aspect for future work.

\begin{figure}[H]
\begin{subfigure}[b]{0.45\textwidth}
\centering
\includegraphics[width=\linewidth]{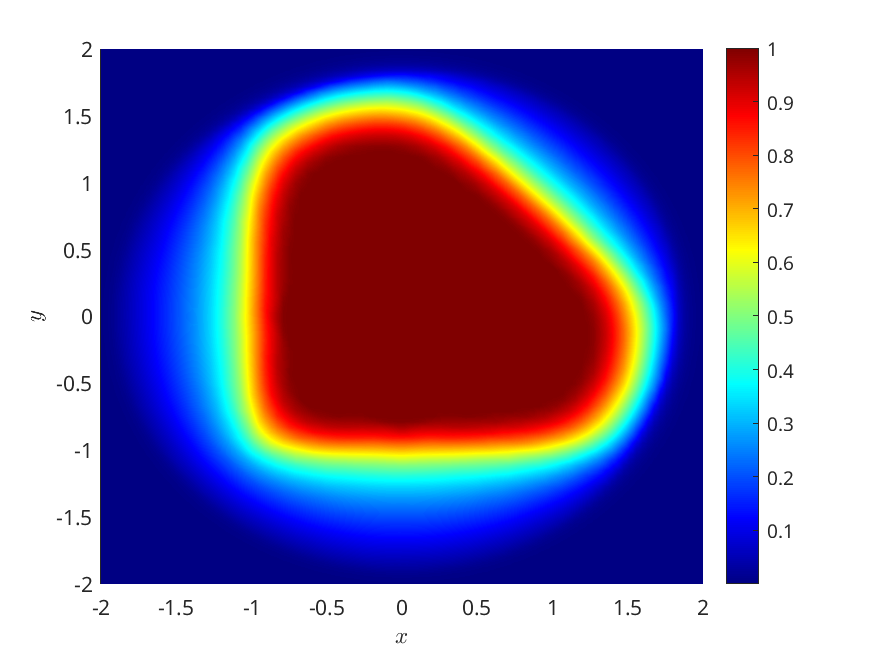} 
\caption{\tiny $G=\{x > -1, y > -1, x+y < 1.8\}$}
\end{subfigure}
\begin{subfigure}[b]{0.45\textwidth}
\centering
\includegraphics[width=\linewidth]{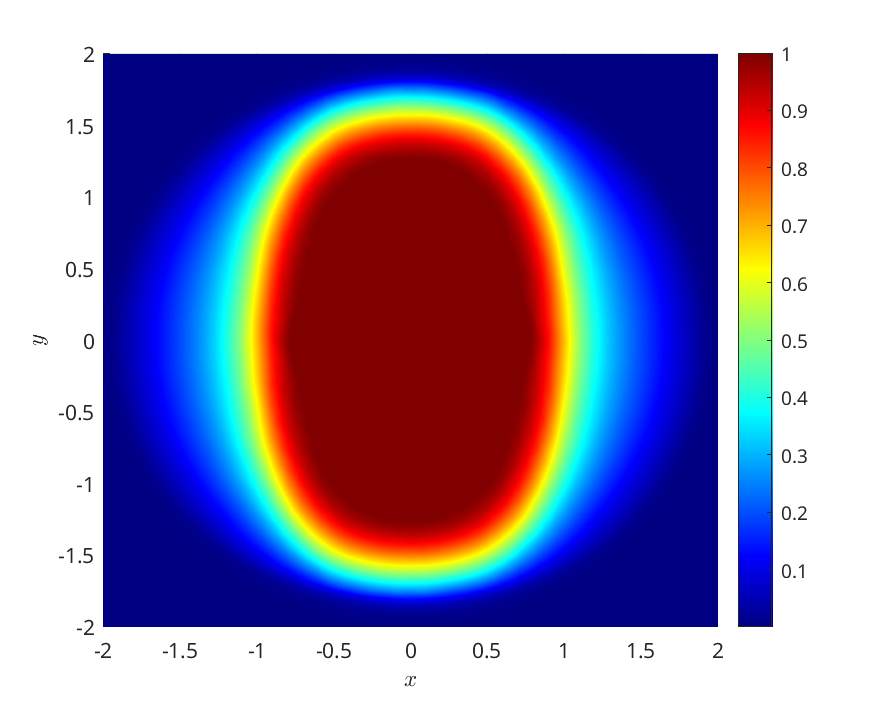}
\caption{\tiny $G=\{x^2+(y/2)^{2}<1\}$}
\end{subfigure}
\begin{subfigure}[b]{0.45\textwidth}
\centering
\includegraphics[width=\linewidth]{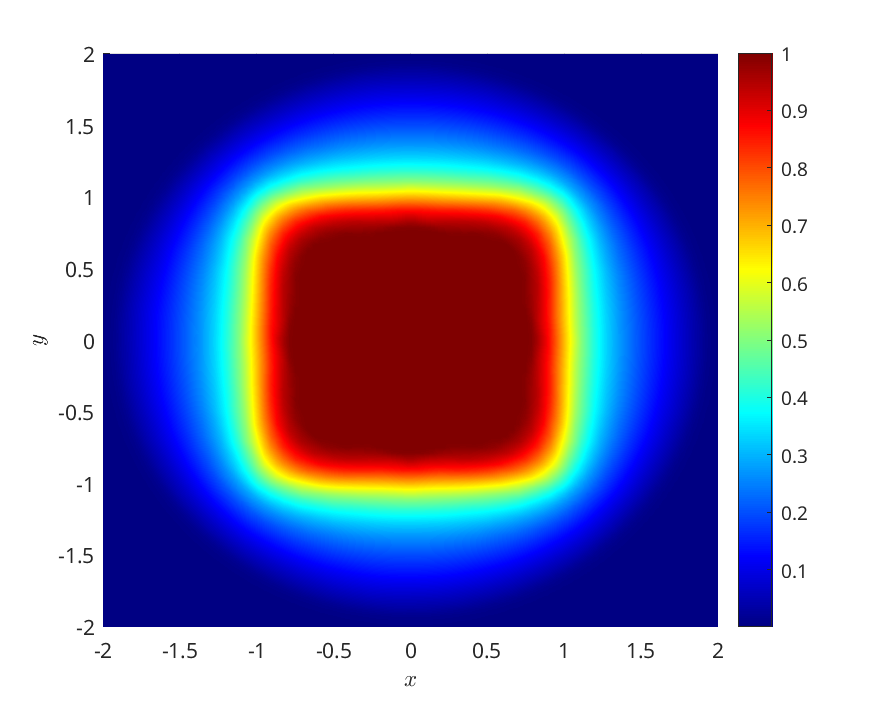}
\caption{\tiny $G=\{-1<x<1,-1<y<1\}$}
\end{subfigure}
\begin{subfigure}[b]{0.45\textwidth}
\centering
\includegraphics[width=\linewidth]{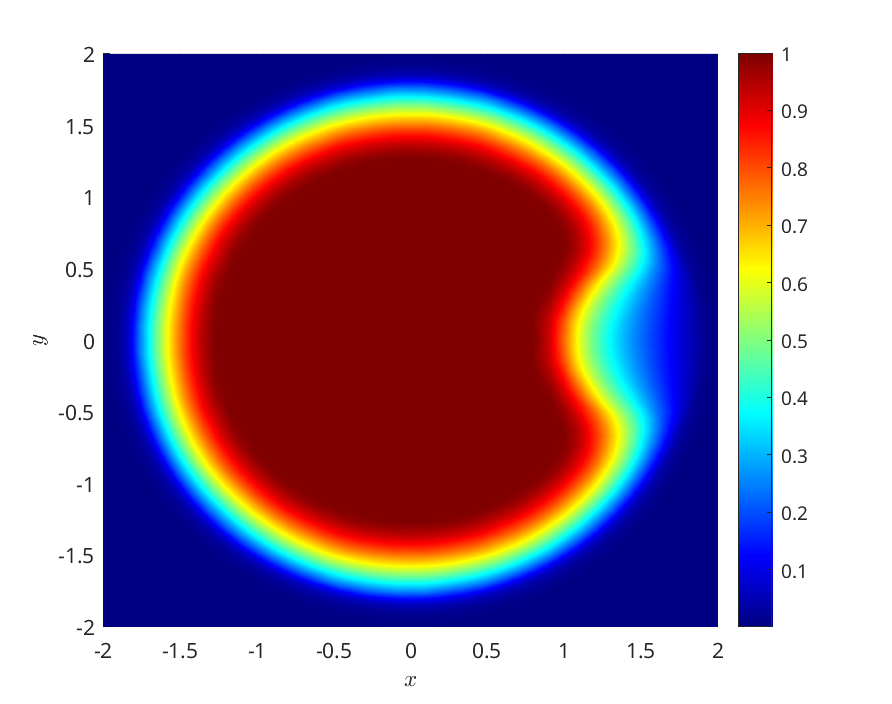}
\caption{\tiny $G=B_{2}\setminus B_{1/2}(3/2)$}
\end{subfigure}

\caption{Approximation of $u_{*}$}
\label{fig:minimizer}
\end{figure}

\appendix 
\crefalias{section}{appendix}

\section{Comparison principle and Harnack inequalities\label{sec:harnack}} 

We make use of the following comparison principle. 

\begin{lemma}\label{lem:comparison-principle}
Let $\Omega$ be an open connected set, and let $V\in L^{\infty}(\Omega)$ satisfy $V\ge 0$ in $\Omega$. If $u\in H^{1}(\Omega)$ satisfies $-\Delta u + Vu\ge 0$ in $\Omega$ and $u|_{\partial\Omega}\ge 0$ (in the sense that $u_{-}\in H_{0}^{1}(\Omega)$), then $u\ge 0$ in $\Omega$. If, in addition, $u\in C(\Omega)$, then either $u\equiv 0$ or $u>0$ in $\Omega$. 
\end{lemma}

\begin{proof}[Sketch of the proof]
Since $u_{-}\in H_{0}^{1}(\Omega)$, we see that 
\begin{equation*}
0 \le \int_{\Omega} (\nabla u\cdot\nabla u_{-} + Vuu_{-})\,\rmd x = -\int_{\Omega} (\abs{\nabla u_{-}}^{2} + V\abs{u_{-}}^{2})\,\rmd x,  
\end{equation*}
showing that $u_{-}\equiv 0$ in $\Omega$, that is, $u\ge 0$ in $\Omega$. 

We now make the additional assumption that $u$ is continuous. In this case, the set $\Omega_{0}:=\{x\in\Omega:u(x)=0\}$ is relatively closed in $\Omega$. If $\Omega_{0}=\emptyset$, then $u>0$ in $\Omega$. Otherwise, let $y\in \Omega_{0}$, then using a version of the Harnack inequality \cite[Theorem~8.18]{GT01Elliptic}, it follows $B_{R}(y)\subset\Omega_{0}$ for some $R>0$ by the arguments in \cite[Theorem~8.19]{GT01Elliptic}, therefore the set $\Omega_{0}$ is relatively open in $\Omega$. Since $\Omega$ is connected and $\Omega_{0}\neq\emptyset$, we conclude that $\Omega=\Omega_{0}$, that is, $u\equiv 0$ in $\Omega$. 
\end{proof}

For clarity, here we recall a variant of the Harnack inequality. 

\begin{lemma}[{\cite[Lemma~2.4]{GS96FreeBoundaryPotential}}]\label{lem:Harnack}
Assume $w\in H^{1}(B_{r})$, $w\ge 0$ on $\partial B_{r}$. 
\begin{enumerate}
\renewcommand{\labelenumi}{\theenumi}
\renewcommand{\theenumi}{\rm (\alph{enumi})} 
\item \label{itm:Harnack-a} If there exists a constant $M_{1}\ge 0$ such that $\Delta w \le M_{1}$ in $B_{r}$, then 
\begin{equation*}
w(x) \ge r^{d} \frac{r-\abs{x}}{(r+\abs{x})^{d-1}}\left(\frac{1}{r^{2}} \dashint_{\partial B_{r}} w \,\rmd S - \frac{2^{d-1}M_{1}}{d} \right) \quad \text{for all $x\in B_{r}$.} 
\end{equation*} 
\item \label{itm:Harnack-b} If there exists a constant $M_{2} \ge 0$ such that $\Delta w \ge -M_{2}$ in $B_{r}$, then 
\begin{equation*}
w(x) \le r^{d} \frac{r+\abs{x}}{(r-\abs{x})^{d-1}}\left(\frac{1}{r^{2}} \dashint_{\partial B_{r}} w \,\rmd S + \frac{M_{2}}{2d} \right) \quad \text{for all $x\in B_{r}$.} 
\end{equation*}
\item \label{itm:Harnack-c} If there exists a constant $M_{3}\ge 0$ such that 
$\abs{\Delta w}\le M_{3}$ in $B_{r}$, then 
\begin{equation*}
\abs{\nabla w(0)} \le  d\left(\frac{1}{r}\dashint_{\partial B_{r}} w \,\rmd S + \frac{M_{3}}{d+1}r\right) . 
\end{equation*}
\end{enumerate} 
\end{lemma}

\begin{remark}
We note that the function $w(x)$ appearing in \ref{itm:Harnack-a} and \ref{itm:Harnack-b} is defined pointwise by choosing its lower/upper semicontinuous representative, following \cite[Theorem~4.1.8]{Hoermander_book_1}. 
\end{remark}

\begin{remark}
Since the original reference \cite{GS96FreeBoundaryPotential} lacks the proof, we briefly outline it for the case $w\in C(\overline{B_r})$, which suffices for our purposes. Here we also slightly refine the estimate in \ref{itm:Harnack-c}. 
We begin by recalling the classical Poisson integral formula \cite[Theorem~2.6]{GT01Elliptic}:
\begin{equation*}
P[\varphi](x):=
\left\{
\begin{aligned} 
& r^{d-2}(r^2-|x|^2)\dashint_{\partial B_{r}}\frac{\varphi(y)}{|x-y|^d}\,\rmd S_y \ , && \text{for all $x\in B_{r}$,} \\
& \varphi(x) , && \text{for all $x\in \partial B_{r}$,}
\end{aligned} 
\right.
\end{equation*}
which belongs to $C^2(B_{r})\cap C(\overline{B_{r}})$  and satisfies $\Delta P[\varphi] = 0$ in $B_{r}$. Applying the weak maximum principle \cite[Theorem~8.1]{GT01Elliptic} to a superharmonic function $w(x)-P[w|_{\partial B_r}](x)+\frac{M_1}{2d}(r^2-|x|^2)$, we obtain
\begin{equation*}
  w(x)\geq P[w|_{\partial B_r}](x)-\frac{M_1}{2d}(r^2-|x|^2) \geq \frac{r^{d-2}(r^2-|x|^2)}{(r+|x|)^d}\dashint_{\partial B_{r}}w\,\rmd S_y-\frac{M_1}{2d}(r^2-|x|^2)
\end{equation*}
since $w\geq 0$. This yields \Cref{lem:Harnack}\ref{itm:Harnack-a}. Similarly, applying the weak maximum principle to a subharmonic function $w(x)-P[w|_{\partial B_r}](x)-\frac{M_2}{2d}(r^2-|x|^2)$, we obtain
\begin{equation*}
  w(x)\leq P[w|_{\partial B_r}](x)+\frac{M_2}{2d}(r^2-|x|^2) \leq \frac{r^{d-2}(r^2-|x|^2)}{(r-|x|)^d}\dashint_{\partial B_{r}}w\,\rmd S+\frac{M_2}{2d}(r^2-|x|^2),
\end{equation*}
which yields \Cref{lem:Harnack}\ref{itm:Harnack-b}. For \Cref{lem:Harnack}\ref{itm:Harnack-c}, we may assume $w\in C^1(B_r(0))$ since $\Delta w\in L^\infty(B_r)$. Using the Green function $G(x,y)$ of $\Delta$ with a Dirichlet boundary condition, we compute: 
\begin{equation*}
\nabla w(0)=\int_{B_r(0)}\nabla_x G(0,y)\Delta w(y)\,\rmd y+\nabla  P[w|_{\partial B_r}](0). 
\end{equation*}
Using the explicit formula for $G(x,y)$ (\cite[(2.23)]{GT01Elliptic}), we obtain 
\begin{align*}
|\nabla w(0)|&\leq\frac{1}{d|B_{1}|}\left|\int_{B_{r}}y\left(\frac{1}{|y|^d}-\frac{1}{r^d}\right)\Delta w(y)\rmd y\right|+\left|\frac{d}{r^2}\dashint_{\partial B_{r}}yw(y)\,\rmd S_y\right|\\
&\leq d\left(\frac{M_{3}}{d+1}r+\frac{1}{r}\dashint_{\partial B_{r}}w\,\rmd S\right),
\end{align*}
which yields \Cref{lem:Harnack}\ref{itm:Harnack-c}.
\end{remark}

We also provide a short proof of Hopf's lemma for weak solutions of superharmonic functions using the mean value inequality for superharmonic functions (\Cref{lem:Harnack}\ref{itm:Harnack-a} with $M_{1}=0$). A version of Hopf's lemma for strong solutions can be found in \cite[Lemma~3.4]{GT01Elliptic}.

\begin{corollary}[Hopf's lemma]\label{cor:Hopf} 
Suppose that $G_{1}$ is a bounded domain satisfying the interior sphere condition at $x_{0}\in\partial G_{1}$, that is, there exist $a\in G_{1}$ and $r>0$ such that $B_{r}(a)\subset G_{1}$ and $\partial B_{r}(a)\cap\partial G_{1}=\{x_{0}\}$. Let $\nu$ denote the unit vector that points from $a$ to $x_{0}$. If $w$ is $C^{1}$ near $x_{0}$, $\Delta w\le 0$ in $G_{1}$ and $w(x)>w(x_{0})$ for all $x\in \partial B_{r}(a)\setminus\{x_{0}\}$, then $\nu\cdot\nabla w(x_{0})<0$. 
\end{corollary}

\begin{proof}
Without loss of generality, it suffices to prove \Cref{cor:Hopf} in the case $a=0$ and $w(x_{0})=0$. Applying the Harnack inequality for superharmonic functions (\Cref{lem:Harnack}\ref{itm:Harnack-a} with $M_{1}=0$), we obtain 
\begin{equation*}
w(x)\ge r^{d-2}\frac{r-\abs{x}}{(r+\abs{x})^{d-1}} \dashint_{\partial B_{r}} w\,\rmd S \quad \text{for all $x\in B_{r}$}
\end{equation*}
For $\epsilon\in(0,r)$, let $x_{-\epsilon}:=x_{0}-\epsilon\nu\in B_{r}$. Then $\abs{x_{-\epsilon}}=r-\epsilon$, and hence 
\begin{equation*}
w(x_{-\epsilon})\ge r^{d-2}\frac{\epsilon}{(2r-\epsilon)^{d-1}} \dashint_{\partial B_{r}} w\,\rmd S, 
\end{equation*}
which implies 
\begin{equation*}
\frac{w(x_{0})-w(x_{-\epsilon})}{\abs{x_{0}-x_{-\epsilon}}} = -\frac{w(x_{-\epsilon})}{\epsilon} \le - \frac{r^{d-2}}{(2r-\epsilon)^{d-1}} \dashint_{\partial B_{r}} w\,\rmd S.  
\end{equation*}
Since $w\in C^{1}$ in a neighborhood of $x_{0}$, taking the limit as $\epsilon\rightarrow 0_{+}$ yields 
\begin{equation*}
\nu\cdot\nabla w(x_{0}) = \lim_{\epsilon\rightarrow 0_{+}} \frac{w(x_{0})-w(x_{-\epsilon})}{\abs{x_{0}-x_{-\epsilon}}} \le - \frac{1}{2^{d-1}r} \dashint_{\partial B_{r}} w\,\rmd S < 0 
\end{equation*}
because $w>0$ on $\partial B_{r}$. This completes the proof. 
\end{proof}

\section{Geometry of free boundaries\label{sec:FB}} 

We recall the free boundary results \cite[Theorems~1.9 and 1.10]{SS25vanishingcontrast} (with $m=0$) in the following proposition:

\begin{proposition}\label{prop:FB}
Let $D\subset\mR^{d}$ be a Lipschitz domain such that $0\in\partial D$, and let $f(0)>0$, $f\in C^{\alpha}(B_{1})$ for some $0<\alpha<1$. Suppose that $u\in H_{\rm loc}^{1}(B_{1})$ solves 
\begin{equation*}
\Delta u = f\chi_{D} \text{ in $\mD'(B_{1})$} ,\quad u|_{B_{1}\setminus\overline{D}}=0. 
\end{equation*}
\begin{enumerate}
\item [(a)] If $d=2$ and $\partial D$ is piecewise $C^{1}$, then $\partial D$ is $C^{1}$ near $0$. 
\item [(b)] If $d=2$ and $D$ is convex, then $\partial D$ is $C^{1}$ near $0$. 
\item [(c)] If $d\ge 3$, then $0\in\partial D$ is not an edge point described in \cite[Definition~1.3]{SS25vanishingcontrast}. 
\end{enumerate}
\end{proposition}

\section{Some properties of Bessel Functions\label{sec:Bessel}} 

For the reader's convenience, this appendix summarizes several properties of Bessel functions. 
Let $K_{\upsilon}$ denoted the modified Bessel function of the second kind of order $\upsilon\in\mR$, which is a real-valued function satisfying 
\begin{equation}
\lim_{t\rightarrow+\infty}\frac{K_{\upsilon}(t)}{t^{-\frac{1}{2}}e^{-t}}=\sqrt{\frac{\pi}{2}}, \label{eq:exponential-decay-K}
\end{equation}
see \href{https://dlmf.nist.gov/10.25.E3}{\texttt{DLMF:10.25.E3}}. 
By using \href{https://dlmf.nist.gov/10.30.E2}{\texttt{DLMF:10.30.E2}} and \href{https://dlmf.nist.gov/10.30.E3}{\texttt{DLMF:10.30.E3}}, we see that 
\begin{equation}
\lim_{t\rightarrow 0_{+}}\frac{K_{\upsilon+1}(t)}{K_{\upsilon}(t)} = +\infty \quad \text{for all $\upsilon\ge0$}. \label{eq:BesselK-limit0}
\end{equation}
By using the facts
\begin{equation}
\begin{aligned}
& \frac{\rmd}{\rmd t}\left(t^{-\upsilon}K_{\upsilon}(t)\right) = -t^{-\upsilon}K_{\upsilon+1}(t) \quad \text{for all $t>0$ and $\upsilon\in\mR$,} \\ 
& \frac{\rmd}{\rmd t}\left(t^{\upsilon}K_{\upsilon}(t)\right) = -t^{\upsilon}K_{\upsilon-1}(t) \quad \text{for all $t>0$ and $\upsilon\in\mR$,}
\end{aligned} \label{eq:differentiation-Bessel}
\end{equation}
which can be found in \href{https://dlmf.nist.gov/10.29.E4}{\texttt{DLMF:10.29.E4}}, we see that  
\begin{equation}
\begin{aligned}
& \Delta \left( \abs{x}^{-\frac{d-2}{2}} K_{\frac{d-2}{2}}(\alpha\abs{x}) \right) = \frac{1}{\abs{x}^{d-1}}\frac{\rmd}{\rmd\abs{x}} \left( \abs{x}^{d-1}\frac{\rmd}{\rmd\abs{x}}\left(\abs{x}^{-\frac{d-2}{2}}K_{\frac{d-2}{2}}(\alpha\abs{x})\right) \right) \\
& \quad = -\alpha \frac{1}{\abs{x}^{d-1}}\frac{\rmd}{\rmd\abs{x}} \left( \abs{x}^{\frac{d}{2}}K_{\frac{d}{2}}(\alpha\abs{x}) \right) = \alpha^{2}\abs{x}^{-\frac{d-2}{2}} K_{\frac{d-2}{2}}(\alpha\abs{x}). 
\end{aligned} \label{eq:screened-Poissson} 
\end{equation}
By using the facts
\begin{equation*}
\begin{aligned}
& \frac{\rmd}{\rmd t}\left(t^{-\upsilon}I_{\upsilon}(t)\right) = t^{-\upsilon}I_{\upsilon+1}(t) \quad \text{for all $t>0$ and $\upsilon\in\mR$,} \\ 
& \frac{\rmd}{\rmd t}\left(t^{\upsilon}I_{\upsilon}(t)\right) = t^{\upsilon}I_{\upsilon-1}(t) \quad \text{for all $t>0$ and $\upsilon\in\mR$,}
\end{aligned} 
\end{equation*}
we see that the identity in \eqref{eq:screened-Poissson} also hold if $K_{\upsilon}$ is replaced by the modified Bessel function of the first kind $I_{\upsilon}$.

Let $J_{\upsilon}$ denoted the Bessel function of the first kind of order $\upsilon\in\mR$, which is a real-valued function, see \href{https://dlmf.nist.gov/10.2.E2}{\texttt{DLMF:10.2.E2}}. By using the facts
\begin{equation}
\begin{aligned}
& \frac{\rmd}{\rmd t}\left(t^{-\upsilon}J_{\upsilon}(t)\right) = -t^{-\upsilon}J_{\upsilon+1}(t) \quad \text{for all $t>0$ and $\upsilon\in\mR$,} \\ 
& \frac{\rmd}{\rmd t}\left(t^{\upsilon}J_{\upsilon}(t)\right) = t^{\upsilon}J_{\upsilon-1}(t) \quad \text{for all $t>0$ and $\upsilon\in\mR$,}
\end{aligned} \label{eq:differentiation-Bessel-J}
\end{equation}
which can be found in \href{https://dlmf.nist.gov/10.6.E6}{\texttt{DLMF:10.6.E6}}, we see that 
\begin{equation}
\begin{aligned}
& \Delta \left( \abs{x}^{-\frac{d-2}{2}} J_{\frac{d-2}{2}}(\beta\abs{x}) \right) = \frac{1}{\abs{x}^{d-1}}\frac{\rmd}{\rmd\abs{x}} \left( \abs{x}^{d-1}\frac{\rmd}{\rmd\abs{x}}\left(\abs{x}^{-\frac{d-2}{2}}J_{\frac{d-2}{2}}(\beta\abs{x})\right) \right) \\
& \quad = -\beta \frac{1}{\abs{x}^{d-1}}\frac{\rmd}{\rmd\abs{x}} \left( \abs{x}^{\frac{d}{2}}J_{\frac{d}{2}}(\beta\abs{x}) \right) = -\beta^{2}\abs{x}^{-\frac{d-2}{2}} J_{\frac{d-2}{2}}(\beta\abs{x}). 
\end{aligned} \label{eq:Helmholtz}  
\end{equation}
The identities \eqref{eq:differentiation-Bessel-J} and \eqref{eq:Helmholtz} also hold if $J_{\upsilon}$ is replaced by the Bessel function of the second kind $Y_{\upsilon}$. We will also need the following facts: 
\begin{equation*}
\lim_{t\rightarrow 0_{+}}\frac{Y_{0}(t)}{\log t} = \frac{2}{\pi} ,\quad \lim_{t\rightarrow 0_{+}} t^{\upsilon}Y_{\upsilon}(t) = -\frac{2^{\upsilon}\Gamma(\upsilon)}{\pi} \text{ for all $\upsilon>0$}
\end{equation*}
and 
\begin{equation*}
\lim_{t\rightarrow 0_{+}} t^{-\upsilon}J_{\upsilon}(t) = \frac{1}{2^{\upsilon}\Gamma(\upsilon+1)} \quad \text{for all $\upsilon\ge 0$,} 
\end{equation*}
see \href{https://dlmf.nist.gov/10.7.E3}{\texttt{DLMF:10.7.E3}}. 
According to \href{https://dlmf.nist.gov/10.21.E2}{\texttt{DLMF:10.21.E2}}, the first positive zero of $J_{\upsilon}$, denoted as $j_{\upsilon,1}$, is strictly smaller than the first positive zero of $J_{\upsilon+1}$, therefore the continuous mapping $t\in(0,j_{\upsilon,1}) \mapsto \frac{J_{\upsilon+1}(t)}{J_{\upsilon}(t)}$ satisfies 
\begin{equation}
\lim_{t\rightarrow 0_{+}} \frac{J_{\upsilon+1}(t)}{J_{\upsilon}(t)}=0 ,\quad \lim_{t\rightarrow j_{\upsilon,1}} \frac{J_{\upsilon+1}(t)}{J_{\upsilon}(t)} = +\infty \quad \text{for all $\upsilon\ge 0$.} \label{eq:quotient-BesselJ}
\end{equation}
We will also need the following Wronskian, which can be found in \href{https://dlmf.nist.gov/10.5.E2}{\texttt{DLMF:10.5.E2}}: 
\begin{equation}
J_{\upsilon}(t)Y_{\upsilon+1}(t) - J_{\upsilon+1}(t)Y_{\upsilon}(t) = -\frac{2}{\pi t} \quad \text{for all $\upsilon\ge 0$ and $t > 0$.} \label{eq:Wronskian-Bessel}
\end{equation}

\end{sloppypar}

\bibliographystyle{custom}
\bibliography{ref}
\end{document}